\documentclass{amsart}

\usepackage{amssymb, amsmath, esint, xcolor}
\usepackage[backref=page]{hyperref}

\renewcommand{\[}{\begin{equation}\begin{aligned}}
\renewcommand{\]}{\end{aligned} \end{equation}}

\usepackage{verbatim}

\newtheorem{thm}{Theorem}
\newtheorem{prop}[thm]{Proposition}
\newtheorem{lemma}[thm]{Lemma}

\newcommand{\x}{\mathbf{x}}

\theoremstyle{remark}
\newtheorem{remark}[thm]{Remark}

\theoremstyle{definition}
\newtheorem{definition}[thm]{Definition}

\author{G\'abor Sz\'ekelyhidi}
\address{Department of Mathematics, Northwestern University, Evanston,
  IL, USA}
\email{gaborsz@northwestern.edu}

\title[Generic neck pinches]{Generic neck pinch singularities along 2D Lagrangian mean
  curvature flow}
\date{}

\begin{document}

\begin{abstract}
  We introduce a notion of nondegenerate neck pinch singularity along the
  Lagrangian mean curvature flow of surfaces in a Calabi-Yau
  surface. We show that such singularities can occur, are stable under
  small perturbations, and any neck pinch singularity can be perturbed
  to such a nondegenerate singularity near the singular time. Using this we
  answer some questions raised by Neves and Joyce. We also introduce
  nondegenerate teardrop singularities and show that these cannot
  occur for embedded flows. 
\end{abstract}

\maketitle

\section{Introduction}
An important problem in complex geometry is to construct special
Lagrangian submanifolds of a Calabi-Yau manifold $X$, and an ambitious proposal
by Thomas-Yau~\cite{Thomas01, TY02} suggests that for a given Lagrangian
$L\subset X$, we can deform $L$ into a special Lagrangian submanifold
when $L$ is ``stable'' in the sense of geometric invariant
theory. Since special Lagrangian submanifolds are area minimizing in
their homology class, a natural approach, studied by Thomas-Yau, is to
consider the mean curvature flow
\[ \frac{d}{dt} L_t = H(L_t), \]
with initial condition given by $L$. Indeed, it is known by
Smoczyk~\cite{Smoczyk} that the Lagrangian condition is preserved
under this flow. Since not all Lagrangian homology classes
admit special Lagrangian representatives (see Wolfson~\cite{Wol05}),
singularities are expected to form. Indeed, Neves~\cite{Neves13}
showed that for any Lagrangian $L$ we can find arbitrarily small Hamiltonian
perturbations $L'$ such that the mean curvature flow with initial
condition $L'$ leads to singularities.

Joyce~\cite{JoyceEMS} proposed a far reaching refinement of the
Thomas-Yau conjecture, relating the
appearance of different singularities along the flow to the behavior of the
corresponding Lagrangians in the derived Fukaya category
$D^b\mathcal{F}(X)$. Roughly speaking, Joyce conjectured that any Lagrangian with
unobstructed Floer homology can be decomposed in a certain sense into
special Lagrangians, using the mean curvature flow through
singularities. Two basic types of singularities
considered by Joyce are the following. For simplicity we focus on the
case when $\dim_{\mathbb{C}}X =2$, although the  conjectural picture applies
in all dimensions. 
\begin{itemize}
  \item {\bf Neck pinches}: these are singularities with tangent flow
    given by the special Lagrangian union $P_1\cup P_2$ of two
    transverse planes. It was shown in \cite{LSSz22} that such
    singularities are modeled on the
    family of complex hypersurfaces $zw = \epsilon$ in $\mathbb{C}^2$,
    as $\epsilon \to 0$. They are the basic mechanism for a Lagrangian
    decomposing into several Lagrangians, or an embedded Lagrangian
    becoming immersed, also considered by
    Thomas-Yau~\cite{TY02} (see also \cite{LSz24}).
  \item {\bf Teardrops}: these singularities have tangent flows
    $P_1\cup_\ell P_2$ given
    by two planes meeting along a line, so they are not special
    Lagrangian. Joyce conjectures~\cite[Conjecture 3.33]{JoyceEMS}
    that such singularities can occur if the Floer homology theory of the Lagrangian becomes
    obstructed as one approaches the singular time. See \cite[Figure
    3.10]{JoyceEMS} for justification of the ``teardrop'' name. Note
    that at present we are not able to show that such singularities
    actually appear. 
  \end{itemize}
An optimistic conjecture is that in two dimensions these are the only
singularities that appear generically, and therefore it is important
to understand them better.

Let us remark here on the much more well studied setting of the mean
curvature flow of surfaces in $\mathbb{R}^3$. By the breakthrough
works of Colding-Minicozzi~\cite{CM12}, 
Bamler-Kleiner~\cite{BK23} and
Chodosh-Choi-Mantoulidis-Schulze~\cite{CCMS24, CCMS24_2}, it is know
that for generic initial closed embedded surfaces in $\mathbb{R}^3$
the mean curvature flow only de\-ve\-lops spherical and cylindrical
singularities, as conjectured by Huisken. Cylindrical singularities 
have been studied extensively, and in particular several authors have
studied notions of nondegenerate neck pinch singularities (see
\cite{GS09,SS20,SWX25}), which correspond to locally decomposing the surface into
two components. 

Our main results pertain to neck pinches along the Lagrangian mean
curvature flow of surfaces. In the setting
of equivariant mean curvature flow in $\mathbb{C}^n$,
Neves~\cite{Neves07} gave initial conditions that lead to a finite
time singularity, while Wood~\cite{Wood24} showed that
these singularities have tangent flows given by $P_1\cup
P_2$. See also Lotay-Oliveira~\cite{LotayOliveira} for neck pinches
with circle symmetry. For the flow in a compact ambient space, Neves~\cite{Neves13}
built on these equivariant examples to show the existence of finite time
singularities, but it was not clear what the structure of these
singularities is -- we will discuss this in more detail below. As such
we do not yet have examples of finite time neck pinch singularities
forming in compact ambient spaces  (for infinite time singularities see
Su-Tsai-Wood~\cite{STW24}). Furthermore we do not know what their
finer properties are in the 
context of Joyce's conjecture. Resolving these questions is our main
goal in this paper. 

We will study a notion of nondegenerate neck pinch
singularity. For the precise notion see Definition~\ref{defn:nondeg},
but let us describe it informally here. Since the (unique) tangent
flow at a neck pinch singularity is the union $P_1\cup P_2$ of two
transverse planes, we have a corresponding solution of the rescaled
mean curvature flow $M_\tau$, with $M_\tau \rightharpoonup P_1\cup P_2$ as
$\tau\to \infty$ (see \eqref{eq:RMCF}). The Lagrangian angle $\theta$ satisfies the drift
heat equation along the flow $M_\tau$, and normalizing $\theta$ to
have zero average and weighted $L^2$-norm 1 at different times, we can
extract a limit solution $u_\infty$ of the drift heat equation on
$P_1\cup P_2$. This is simply a pair of solutions $(u_{\infty, 1},
u_{\infty, 2})$ on the two planes. We say that the neck pinch is
nondegenerate if this limit solution is $(-(8\pi)^{-1/2},
(8\pi)^{-1/2})$, and at the same time for sufficiently large $\tau$ the
  surface $M_\tau$ is Hamiltonian isotopic to the Lagrangian connect
  sum $P_1 \# P_2$ (see Section~\ref{sec:connectsum} for more
  details). Note that the $(8\pi)^{1/2}$ constant appears because of
  the $L^2$ normalization. The important point is that the
  limit solution is negative on $P_1$ and positive on
  $P_2$, rather than the other way around. Even more
informally, a nondegenerate neck pinch is locally modeled on a
shrinking Lawlor neck, where the angle between the two asymptotic
planes of the Lawlor neck is increased slightly, similarly to the
equivariant examples of Neves~\cite{Neves07}. 

Our first main result 
is the following. The result bears some relation to works on the mean curvature
flow of surfaces in $\mathbb{R}^3$ such as \cite{GS09,SS20,SWX25}, but the methods
are completely different. 

\begin{thm}\label{thm:perturb}
  Suppose that for $t\in [0,T)$, $L_t$ is a rational, graded Lagrangian mean curvature
  flow in a compact Calabi-Yau surface (see
  Section~\ref{sec:LMCFprelim} for definitions).
  Suppose that the flow develops
  a first time singularity at $(x_0, T)$ with tangent flow given by the
  transverse union of two planes. Then we can find a small Hamiltonian
  perturbation $L_{T'}'$ of $L_{T'}$ for some $T' < T$ close to
  $T$, such that the flow with any initial
  condition sufficiently close to $L_{T'}'$ develops a nondegenerate neck pinch
  singularity at a time $T_0 < T$, and no other singularities at that time.
\end{thm}

In fact it is enough to assume that $L_t$, for $t$ close to $T$, is sufficiently close to a
Lawlor neck at a sufficiently large range of scales for the conclusion
of this result to hold (see Proposition~\ref{prop:singexist} for the precise
statement). As a consequence we can construct examples of
Lagrangian mean curvature flows in Calabi-Yau surfaces developing
nondegenerate neck pinch singularities. This is even possible with an
initial Lagrangian that is Hamiltonian isotopic to a special
Lagrangian and whose Lagrangian angle has arbitrarily small
oscillation. 

\begin{thm}\label{thm:2}
  There is a compact Calabi-Yau surface $X$, such that for
  any $\delta > 0$ we can find a Lagrangian $L_0$ in $X$ with the
  following properties. The Lagrangian $L_0$ is Hamiltonian isotopic to a special
  Lagrangian $L$ in $X$, the Lagrangian angle of $L_0$ has oscillation at
  most $\delta$, and the flow starting with any sufficiently small smooth
  perturbation of $L_0$ exists until a finite time $T
  > 0$ at which time it develops a nondegenerate neck pinch singularity. 
\end{thm} 

As far as we know this gives the
first examples of finite time singularity formation along the
Lagrangian mean curvature flow in a compact Calabi-Yau surface with
almost calibrated initial condition. Note that infinite time neck
pinches were recently constructed by Su-Tsai-Wood~\cite{STW24},
although those are likely to be unstable under small perturbations. 
Almost calibrated singularities were constructed by Neves~\cite{Neves07}
in $\mathbb{C}^2$ under symmetry assumptions, while the finite time
singularities given by Neves~\cite{Neves13} in a compact Calabi-Yau
surface are not almost calibrated, though they are stable under small
perturbations. The question of the existence of singularities shown in
Theorem~\ref{thm:2} was raised by Neves~\cite[Remark 1.3
(5)]{Neves13}.

When a neck pinch singularity forms along the flow at a point $(x_0,
T)$, then it was shown in \cite{LSSz22} that at time $T$ the
Lagrangian is a $C^1$ immersed submanifold, with a self intersection at
$x_0$, and the flow can be continued as an immersed flow (as long as no
other singularities form at time $T$). In Joyce's conjectural picture
it is important to have some further information about the Lagrangian
angles of the two local sheets at $x_0$ for times $t > T$ in order to
define suitable bounding cochains and continue the flow (see
\cite[Section 2.6]{JoyceEMS} for details). Just knowing that the tangent flow
is given by two transverse planes is likely not sufficient to ensure
this additional condition, however we show that nondegenerate
neck pinches satisfy it. 

\begin{thm}\label{thm:Joyceconj}
  Let $L_t\subset X$ be a rational, graded Lagrangian mean curvature
  flow in a compact Calabi-Yau surface, with a nondegenerate neck pinch
  singularity at $(x_0, T_0)$. Suppose that there are no other
  singularities at time $T_0$, and we extend the flow for times slightly
  greater than $T$ as an immersed flow. In a neighborhood $U$ of
  $x_0$, for times $t > T_0$, the flow is given by a union of two $C^1$ Lagrangians
  $L^1_{T_0}\cup L^2_{T_0}$, intersecting transversely at $p(t)$, where
  $p(T_0)=x_0$. The ordering of these two sheets is
  chosen so that for $t < T_0$ close to $T_0$, inside $U$, the Lagrangian
  $L_t$ is given by the Lagrangian connect sum $L^1_{T_0}\# L^2_{T_0}$.

  Given these conditions the Lagrangian angles after the singular time
  satisfy $\theta_{L^2_t}(p(t)) > \theta_{L^1_t}(p(t))$ for $t > T_0$
  sufficiently close to $T_0$, as conjectured by Joyce~\cite[Conjecture
  3.16(v)]{JoyceEMS}. 
\end{thm}

\begin{remark}
  In the setting of Theorem~\ref{thm:2} we can continue the flow as a
  flow $L_t$ of immersed Lagrangians for time $t$ after the neck
  pinch. In line with Joyce's conjecture we would expect that
  eventually $L_t$ converges to the original special Lagrangian
  $L$. For this we would need to perform an ``opening the neck''
  surgery discussed in \cite[Problem 3.14]{JoyceEMS}, and studied
  by Begley-Moore~\cite{BM17}. It is an interesting problem to show
  that after such an additional surgery the flow will indeed converge to the
  initial special Lagrangian $L$. 
\end{remark}

The basic geometric idea in the proofs of Theorems~\ref{thm:perturb},
\ref{thm:2} comes from Neves's
construction~\cite{Neves13} of finite time singularities, and our
argument gives more information about the singularity formation in his
setting. We briefly
recall Neves's setup, using some of the notation of \cite{Neves13}. We start with
a smooth Lagrangian submanifold $\Sigma$ in a compact Calabi-Yau surface
$X$, and let $x\in \Sigma$. Choosing a different holomorphic volume form if
necessary, and scaling up the metric on $X$, we can arrange that a
very large neighborhood $U$ of
$x$ can be identified with a large ball $B$ in $\mathbb{C}^2$. Under this 
identification we can ensure that the symplectic forms agree, $\Sigma$
coincides with the $x_1x_2$-plane, while the Calabi-Yau structures
agree up to a small error. Neves defines a new Lagrangian $L$, which
agrees with $\Sigma$ outside of $U$, while inside $U\cong B$ it is given by
the circle invariant Lagrangian
\[ L'\cap B = \{ (\gamma(t) \cos\alpha, \gamma(t) \sin\alpha)\, :\,
  \alpha\in S^1, \quad t\geq 0\} \cap B, \]
where $\gamma: [0,\infty)\to \mathbb{C}$ is a suitable profile
curve. The profile curve is chosen so that for a smaller ball
$B'\subset B$ the Lagrangian $L\cap B'= Q^1\cup Q^2$ has two
connected components (see \cite{Neves13} for details). One
component, $Q^2$, is simply a plane, while $Q^1$ is the graded connect
sum $P_1\# P_2$ of two planes. In fact we can arrange that $Q^1$ is a
perturbation of a Lawlor neck as in Definition~\ref{defn:model2}.

Neves shows (see \cite[Proposition 4.6]{Neves13} as well as
Proposition~\ref{prop:Neves} below), that with such an initial Lagrangian $L$, if no
singularity forms, then for $t\in [1,2]$ the flow inside the smaller
ball $B'$ will be well approximated by the union of the flows $Q^1_t$
and $Q^2_t$ of the two separate components. The flow $Q^2_t$ will
remain close to the plane $Q^2$, while $Q^1_t$ will be well
approximated by the unique smooth expander $\sqrt{t}Q$ asymptotic to
$P_1\cup P_2$. By work of Anciaux~\cite{Anc06} (see also more
generally Lotay-Neves~\cite{LN13} and Imagi-Joyce-Dos
Santos~\cite{IJS16}) this expander is the Lagrangian connect sum
$P_2\# P_1$. Since $Q^1$ is the reversed connect sum $P_1\# P_2$, and
these two connect sums are not Hamiltonian isotopic (see
Theorem~\ref{thm:Seidel} below), Neves's expectation is that a
singularity must have formed. Since the flow may a priori fail to
remain embedded inside $B'$,  it was not known that the flow
defines an (ambient) Hamiltonian isotopy. Because of this, 
Neves had to use further indirect arguments to show that a singularity 
must form, and as a result was also not able to describe the type of
singularity that appears. We will show the following, verifying Neves's
expectation. 

\begin{thm}\label{thm:Neves}
  Suppose that in the setup described above, the Lagrangian $Q^1$
  satisfies the conditions in Proposition~\ref{prop:singexist}.
  Then the flow remains embedded inside
  $B'$ until the first singular time $T\in (0,1)$, and at time $T$ the
  component $Q^1_t$ develops a neckpinch singularity.
\end{thm}

A consequence of this is that at time $T$ the Lagrangian $L_T$ decomposes
as the union of two Lagrangians,  $L_T = L_T^1\cup L_T^2$, where
$L_T^1$ is an immersed Whitney sphere, while $L_T^2$ is a small
perturbation of the original Lagrangian $\Sigma$. This behavior was
conjectured by Joyce~\cite[Example 3.28]{JoyceEMS}.

We now briefly describe the proof of Theorem~\ref{thm:Neves}, which is
closely related to the proofs of Theorems~\ref{thm:perturb} and
\ref{thm:2}. Keeping the notation from above, we study the flow
$Q^1_t$ of the component $Q^1$ inside the ball $B'$. Recall that in
this ball $Q^1$ is a small perturbation of a Lawlor neck, with the
angle between the two asymptotic planes slightly increased (see
Definition~\ref{defn:model2} for the details). The main new ingredient is to
show that $Q^1_t$ remains embedded for times $t < 2$ as long as it is
smooth. To see this, let us define $T_{emb}$ to be the supremum of
times for which $Q^1_t$ is smooth and embedded, so that at $t=T_{emb}$
either the flow becomes immersed, or a singularity forms. From Neves's
results we know that necessarily $T_{emb} < 1$, since otherwise $P_1\#
P_2$ would be Hamiltonian isotopic to $P_2\# P_1$.  

There must be a point $(x, T_{emb})$ where the density of the flow
$Q^1_t$ is at least 2, since the flow either forms a singularity
(and in two dimensions the density of a singular point is at least 2),
or there is an immersed point. From Huisken's monotonicity formula we
find that the Gaussian area of $Q^1$ at scale $\sqrt{T_{emb}}$,
centered at $x$ is also at least $2-\kappa$ for a $\kappa > 0$ that
can be taken arbitrarily small by choosing the initial data $Q^1$
suitably. Using the choice of $Q^1$, we will show that
$T_{emb}^{-1/2}(Q^1-x)$ is well approximated by the special Lagrangian
union $P_1\cup P_2$. We can then invoke the main result of
Lotay-Schulze-Sz\'ekelyhidi~\cite{LSSz22}, which implies that the
tangent flow at $(x, T_{emb})$ is also a transverse union of two
planes, and in particular locally $Q^1_{T_{emb}}$ is the union of two
$C^1$ submanifolds intersecting transversely at $x$. Since $Q^1_t$ is
smooth and embedded for $t < T_{emb}$, this is only possible if a
singularity forms at $T_{emb}$. We have therefore shown that $Q^1_t$
remains embedded until the first singular time, which is at most
1. Therefore the construction also implies that this first singularity
must be a neckpinch, as claimed.

We will also briefly consider teardrop singularities for the
Lagrangian mean curvature flow of surfaces. These are singularities
where tangent flows are given by the union $P_1\cup_\ell P_2$ of two
planes meeting along a line $\ell$. Uniqueness of such tangent flows
is not known, however it follows from the work of
Neves~\cite{NevesSurvey} (see \cite[Proposition 2.7
4.3]{LSSzAncient}) that if one tangent flow is of this form, then any
other tangent flow must be too, perhaps for different planes meeting
along a line.  In \cite[Conjecture 3.33]{JoyceEMS} Joyce proposed 
a mechanism by why non-special Lagrangian tangent flows of this type
could arise, due to the Floer homology of the Lagrangians $L_t$
becoming obstructed. A particular expectation is that such a
non-special Lagrangian tangent flow could only arise if near the
singularity $L_t$ is not embedded. We introduce a notion of
nondegenerate teardrop singularity (see
Definition~\ref{defn:ndteardrop}), and show the following.

\begin{thm}\label{thm:ndteardrop}
  Suppose that for $t\in [0,T)$ the flow $L_t$ is a rational, graded
  Lagrangian mean curvature flow in $\mathbb{C}^2$, with uniform area
  ratio bounds. Suppose that $L_t$ develops a nondegenerate teardrop
  singularity at $(x_0, T)$. Then $L_t$ must have immersed points at
  spacetime points arbitrarily close to $(x_0, T)$. 
\end{thm}

We expect that a similar result holds in a compact Calabi-Yau ambient
space, but we restrict to $\mathbb{C}^2$ here in order to be able to
use the results from \cite{LSSzAncient} on ancient solutions more
directly. Unfortunately at present we do not know whether teardrop
singularities exist, let alone nondegenerate ones. For this, and to
study further properties of such singularities, an important
ingredient would be to develop estimates analogous to the ones in
\cite{LSSz22} used to prove the uniqueness of tangent flows for neck
pinches. 

In the next section we will give some background on the Lagrangian
mean curvature flow and Lagrangian connect sums. In
Section~\ref{sec:singexist} we will give the details of the above
argument, showing that neckpinch singularities must form with suitable
initial conditions. We will further refine this analysis in
Section~\ref{sec:nondeg} where we introduce nondegenerate
singularities, and show that they satisfy the property in
Theorem~\ref{thm:Joyceconj} conjectured by Joyce. In
Section~\ref{sec:teardrop} we will prove
Theorem~\ref{thm:ndteardrop}.

\subsection*{Acknowledgements}
The author thanks Felix Schulze for helpful comments. This work was
supported in part by NSF grant DMS-2506325. 

\section{Preliminaries}
\subsection{Lagrangians and the mean curvature
  flow}\label{sec:LMCFprelim}

Let $X$ be a compact Calabi-Yau surface, equipped with the flat
K\"ahler metric $\omega$, whose volume form is the holomorphic
$2$-form $\Omega = \omega^2/2$. We denote the corresponding
Riemannian metric by $g$, and the complex structure by $J$. A
Lagrangian $L\subset X$ is a two-dimensional, possibly
immersed, submanifold that satisfies $\omega|_L = 0$. We say that $L$
is graded, if it is equipped with a function $\theta : L \to
\mathbb{R}$ that satisfies
\[ \Omega|_L = e^{i\theta} dV_L, \]
where $dV_L$ is the Riemannian volume form. $L$ is special Lagrangian
if we can choose $\theta$ to be constant. Recall that in general we
have $J\nabla\theta = \mathbf{H}$ on $L$, where $\mathbf{H}$ denotes
the mean curvature vector of $L$.

On $\mathbb{C}^2$ we have the Liouville form
\[ \lambda = \sum_{i=1}^2 x_i\, dy_i - y_i \, dx_i, \]
satisfying $d\lambda = 2\omega$. A Lagrangian $L$ is called exact if
$\lambda|_L$ is an exact one-form. 

The (Lagrangian) mean curvature flow is given by a family of
Lagrangians $L_t$ satisfying the evolution equation
\[ \frac{\partial L_t}{dt} = \mathbf{H}(L_t). \]
Following White~\cite{White05}, it is useful to consider an isometric
embedding $X\subset \mathbb{R}^N$, and view the $L_t$ as submanifolds
of $\mathbb{R}^N$ evolving under a mean curvature flow with an
additional term
\[ \frac{\partial L_t}{\partial t} = \mathbf{H} + \nu, \]
where $\nu$ is defined in terms of the second fundamental form of $M$
restricted to the tangent space of $L_t$.

Huisken's monotonicity formula~\cite{Hui90} takes the following form,
for a time dependent family of functions $f$ on $L_t$:
\[ \label{eq:mon} \frac{d}{dt} \int_{L_t} f \rho_{x_0, t_0}\, d\mathcal{H}^2 &=
  \int_{L_t} (\partial_t f - \Delta f) \rho_{x_0, t_0}\,
  d\mathcal{H}^2 + \int_{L_t} f \left|\frac{\nu}{2}\right|^2
  \rho_{x_0,t_0}\, d\mathcal{H}^2 \\
  &\quad - \int_{L_t} f\left| \mathbf{H} -
    \frac{(x-x_0)^\perp}{2(t_0-t)} + \frac{\nu}{2}\right|^2 \rho_{x_0,
    t_0}\, d\mathcal{H}^2. \]

Here $\rho_{x_0, t_0}$ denotes the backwards heat kernel (in 2
dimensions): 
\[ \label{eq:heatkernel} \rho_{x_0, t_0} (x,t) = \frac{1}{4\pi(t_0-t)} \exp\left( -
    \frac{|x-x_0|^2}{4(t_0-t)}\right), \]
for $t < t_0$. In terms of this we also define the Gaussian area of a
Lagrangian $L$ at different scales (still working with the embedding
$L\subset X\subset \mathbb{R}^N$): given $x_0\in \mathbb{R}^N$ and $r_0
> 0$, we let
\[ \Theta(L, x_0, r_0) = \Theta(r_0^{-1}(L-x_0)) = \int_L \rho_{x_0,
    0}(x, -r_0^2)\, d\mathcal{H}^2, \]
where
\[ \Theta(L) = \int_L \rho_{0,0}(x, -1)\, d\mathcal{H}^2. \]

If the flow $L_t$ develops a singularity at a point $(x_0, T_0)$, then
we consider blowup limits of the flow, leading to tangent flows at the
singularity. More precisely, for any sequence $b_k \to \infty$, we
consider the sequence of mean curvature flows
\[ L^k_t := b_k\left( L_{b_k^{-2}t + T_0} - x_0\right). \]
Using the monotonicity formula, it was shown by Huisken~\cite{Hui90}
(see also White~\cite{Whi94}, Ilmanen~\cite{Ilm95}) that along
subsequences this sequence of flows converges in a weak sense to a shrinking solution
of the mean curvature flow, defined for $t\in (-\infty, 0)$ (in
$\mathbb{C}^2$ in our setting). Moreover, for a graded Lagrangian mean
curvature flow in $\mathbb{C}^2$, Neves~\cite{Neves07} showed that the
tangent flows are all given by static flows of unions of Lagrangian
planes. In terms of the tangent flows we say that a singularity is a
neck pinch if a tangent flow is the special Lagrangian union
$P_1\cup P_2$ of two transverse planes, and a teardrop singularity if
a tangent flow is the union $P_1\cup_\ell P_2$ of two planes meeting
along a line. Note that in general it is unknown whether tangent flows
are unique, however for neck pinches this was shown in \cite{LSSz22}.

When studying the tangent flows at a  point $(x_0, T_0)$,
it is convenient to consider the \emph{rescaled} mean
curvature flow, centered at $(x_0, T_0)$ corresponding to a
mean curvature flow $L_t$. For this we consider an isometric embedding
$X\subset \mathbb{R}^N$ as above, and consider the following family of
immersed surfaces in $\mathbb{R}^N$:
\[ \label{eq:RMCF} M_\tau = e^{\tau/2} \left(L_{-e^{-\tau} + T_0} - x_0\right), \]
for $\tau \in [-\log T_0, \infty)$. Note that $M_\tau\subset
e^{\tau/2} X$. The rescaled flow has the property that subsequential limits of
$M_\tau$ as $\tau\to \infty$ are given by the possible tangent flows
of the flow $L_t$ at $(x_0, T_0)$. 

The normal variation of the rescaled flow is given by
\[ \frac{\partial x}{\partial \tau} = \mathbf{H} +
  \frac{x^\perp}{2} + \nu, \]
where $\mathbf{H}$ is the mean curvature of $M_\tau\subset
\mathbb{R}^N$, and $\nu$ depends on the second fundamental form of
$e^{\tau/2}X$. As a result,  $|\nu| \leq Ce^{-\tau/2}$. We will often consider
translations of the rescaled flow $M_\tau$ in time, i.e. flows of the
form $\tau\mapsto M_{\tau + \tau_0}$. In this case the forcing term
above satisfies $|\nu| \leq Ce^{-\tau_0} e^{-\tau}$, which can be made
arbitrarily small for all $\tau > 0$ by choosing $\tau_0$ large. 

In order to recall the main result of \cite{LSSz22}, we need the
following definition, following Fukaya~\cite[Definition 2.2]{Fuk03}.
\begin{definition}\label{defn:rational}
  Suppose that $[\omega]$ defines an integral cohomology class on $X$,
  and let $\xi$ be a complex line bundle equipped with a unitary
  connection $\nabla^\xi$ with curvature form $2\pi i\omega$. If
  $L\subset X$ is Lagrangian, then
  $\nabla^\xi$ defines a flat connection when restricted to $L$. We
  say that $L\subset X$ is rational, if the holonomy group of
  $\nabla^\xi$ on $L$ is a finite subgroup of $U(1)$. 
\end{definition}

In addition we will use the following graphicality notion in several
places.
\begin{definition}
  Let $P_1, P_2 \subset T_0X\subset \mathbb{R}^N$ be two transverse
  planes through the origin. For $c > 0$ and $r < R$ we say that $L\subset
  \mathbb{R}^N$ is $c$-graphical over $L$ on the annulus $A_{r, R} :=
  B_R(0)\setminus \overline{B_r(0)}$, if $L\cap A_{r, R}$ is the graph
  of a vector field $v$ over $(P_1\cup P_2)\cap A_{r, R}$ such that
  \[ |x| |v| + |\nabla v| + |x|^{-1} |\nabla^2 v| + |x|^{-2} |\nabla^3
    v| < c. \]
\end{definition}
Note that if we have a sequence $L_k\subset \mathbb{R}^N$ and $c_k\to
0$ such that $L_k$ is $c_k$ graphical over $P_1\cup P_2$ on $A_{c_k,
  c_k^{-1}}$, then on compact sets of $\mathbb{R}^N\setminus \{0\}$ we
have $L_k \to P_1\cup P_2$ in a natural $C^{2,\alpha}$ sense. 

For a graded Lagrangian $L\subset e^{\tau/2}X \subset \mathbb{R}^N$ we define
the excess (relative to two planes) by
\[  \mathcal{A}(L) = \int_L e^{-|x|^2/4}\, d\mathcal{H}^2 -
  2\int_{\mathbb{R}^2} e^{-|x|^2/4}\, d\mathcal{H}^2 + \inf_{\theta_0}
  \int_{L} |\theta - \theta_0|^2\, e^{-|x|^2/4}\, d\mathcal{H}^2. \]
We can now state the result from \cite{LSSz22} that we will use.

\begin{prop}\label{prop:LSSzunique}
  Suppose that $L_t$ is a rational, graded Lagrangian mean curvature
  flow in a compact Calabi-Yau manifold $X$ as above, and suppose that
  the area ratios of $L_0$ are bounded by $K_0$. 
  Let $\epsilon > 0$. There exists $\delta > 0$, depending on
  $\epsilon$ as well as on the choice of $P_1,P_2$ and $K_0$, with
  the following property. Suppose that we have a rescaled flow
  $M_\tau$ obtained from $L_t$ as above, translated in time so that we
  have the following:
  \begin{enumerate}
    \item along $M_\tau$, the forcing term $|\nu| < \delta$ for all
      $\tau \geq 0$;
    \item on the annulus $A_{\delta, \delta^{-1}}$ the surface $M_0$
      is $\delta$-graphical over $P_1\cup P_2$;
    \item we have the bound $\mathcal{A}(M_0) - \mathcal{A}(M_\tau) <
      \delta$ for all $\tau > 0$;
    \item the surface $M_0 \cap B_1(0)$ is connected, i.e. it is not
      the union of two proper submanifolds.
    \item for any closed curve $\gamma\subset M_0\cap B_1(0)$ we have
      \[ \int_\gamma \lambda = 0, \]
      where $\lambda$ is the Liouville one-form in terms of a Darboux
      chart.
    \end{enumerate}
   Then for all $\tau > 0$ the surface $M_\tau$ is $\epsilon$-graphical
   over $P_1\cup P_2$ on the annulus $A_{\epsilon, \epsilon^{-1}}$. 
 \end{prop}

\subsection{Lawlor necks and connected sums}\label{sec:connectsum}
We first review Lawlor necks in $\mathbb{C}^2$. Recall that in two
complex dimensions a special Lagrangian $L$ is the same as a complex
curve for a hyperk\"ahler rotated complex structure. Using this, it is
simple to classify special Lagrangians in $\mathbb{C}^2$ that are
asymptotic to a pair of planes near infinity. In suitable rotated
complex coordinates (depending on the Lagrangian angle) we can arrange
that the two planes are given by $z=0$ and $w=0$. Up to
translations, special
Lagrangians asymptotic to these planes are given by the equation $zw = t$, for $t\in
\mathbb{C}$. Note that if $t' = ct$ for a real $c > 0$, then the
corresponding Lagrangians are Hamiltonian isotopic, and are related by
scaling. Note that a similar classification also holds in higher dimensions
for exact special Lagrangians asymptotic to two planes, due to
Imagi-Joyce-dos Santos~\cite{IJS16}. 

We now review the Lagrangian connect sum construction due to
Polterovich~\cite{Pol91}. First we consider the connect sum of the
$x_1x_2$-plane $P_1$ with the $y_1y_2$-plane $P_2$ inside
$\mathbb{C}^2$. We choose a smooth
curve $\gamma : \mathbb{R} \to \mathbb{C} \setminus\{0\}$ such that
\begin{itemize}
\item $\mathrm{arg}(\gamma(t)) \in [0,\pi/2]$ for all $t$,
\item $\gamma(t) = -t\sqrt{-1}$ for $t < -1$, and $\gamma(t) = t$ for
  $t > 1$.
\end{itemize}
We then define the Lagrangian connect sum
\[ P_1\# P_2 = \{ (\gamma(t) \cos \alpha, \gamma(t) \sin\alpha)\, \:\,
  \alpha\in \mathbb{R}/2\pi\mathbb{Z}, \quad t\in \mathbb{R}\}. \]

More generally, suppose that $L_1, L_2$ are two Lagrangians
in $X$, intersecting transversely at a point $p$. We can symplectically identify a
neighborhood of $p$ with a neighborhood of the origin in
$\mathbb{C}^2$, such that under this identification we have $L_1 =
P_1$
and $L_2 = P_2$. The Lagrangian connect sum $L_1\# L_2$ is defined by
replacing $L_1$ and $L_2$ in a small neighborhood of the origin by a
suitably scaled down copy of $P_1\# P_2$. One can show that up to
Hamiltonian isotopy the end result is independent of the choice of
$\gamma$ satisfying the conditions above, and the choice of scaling.

Note that outside of a small neighborhood $U$ of $p$, the two Lagrangian connect sums
$L_1\# L_2$ and $L_2 \# L_1$ coincide with $L_1\cup L_2$, and so the
Lagrangian connect sum is a way of smoothing out the singularity of
the transverse union $L_1\cup L_2$ at $p$. An important 
result due to Seidel~\cite{Sei99} is these two
connect sums are not Hamiltonian isotopic, with isotopies supported in
the neighborhood $U$. Note that this plays an important role in the
work of Thomas~\cite{Thomas01} and Thomas-Yau~\cite{TY02} and was also
highlighted by Neves~\cite{Neves13} as a potential reason for
singularity formation.  While the result is not stated in this form,
it follows from combining Theorem 1.1 and Propositions A.2, A.3 in
\cite{Sei99}. 
 
\begin{thm}[Seidel~\cite{Sei99}] \label{thm:Seidel}
  The Lagrangian connect sums $P_1\# P_2$ and $P_2\# P_1$ are not
  Hamiltonian isotopic in $\mathbb{C}^2$ with compactly supported
  isotopies.
\end{thm}

\subsection{The local model}\label{sec:Nkappa}
We now discuss the local model that will lead to singularity
formation. We start with an exact Lawlor neck $N\subset \mathbb{C}^2$,
asymptotic to two planes $P_1\cup P_2$. Note that up to scaling and
translation there
are two inequivalent such Lawlor necks in two dimensions,
corresponding to the connect sums $P_1\# P_2$ and $P_2 \# P_1$. We
assume that $N$ is Hamiltonian isotopic, with an isotopy converging to
the identity at infinity, to $P_1\# P_2$.

For sufficiently small $\kappa > 0$ we define a perturbation
$N_\kappa$ of $N$ as follows. We let $H$ denote a quadratic
polynomial such that $H=0$ on $P_1$, and $\Delta H = 1$ on $P_2$. For
example if $P_1$ is the $x_1x_2$-plane, and $P_2$ is the
$y_1y_2$-plane, we can let $H = \frac{1}{2}y_1^2$. Then for
small $\kappa > 0$ we define the family $N_\kappa$ by the normal
variation
\[ \frac{\partial x}{\partial \kappa} = J\nabla H. \]
This flow exists smoothly for sufficiently small $\kappa > 0$
(depending on the choice of $P_1,P_2$), defining a Lagrangian that is
Hamiltonian isotopic to $N$. We have the following.

\begin{lemma} \label{lem:Nkappa1}
  For sufficiently small $\kappa > 0$ the Lagrangian
  $N_\kappa$ satisfies:
  \begin{itemize}
  \item $N_\kappa$ is asymptotic at infinity to $P_1, P_{2,\kappa}$,
    where $P_{2,\kappa}$ is a perturbation of $P_2$. 
  \item As $\kappa \to 0$, the family $N_\kappa$ converges to $N$
    smoothly on compact sets.
  \item $N_\kappa$ is the graph of a normal vector field $v$ over $N$
    with $|x|^{-1} |v|  + |\nabla v| + |x| |\nabla^2 v| \leq C\kappa$
    for $C$ independent of $\kappa$.
  \item For any $\epsilon > 0$ there is an $R > 0$ such that for $|x|
    > R$ we have $|\theta| < \epsilon\kappa$ on the component of
    $N_\kappa$ corresponding to $P_1$, and $|\theta - \kappa| <
    \epsilon\kappa$ on the component corresponding to $P_{2,\kappa}$. 
    \end{itemize}
\end{lemma}

We will need the following result about scales of $N_\kappa$ at which
its Gaussian area is close to 2.
\begin{lemma}\label{lem:NkappaG}
  For any $\epsilon > 0$ there exists $\kappa_0 > 0$ with the following
  property. Suppose that $\kappa < \kappa_0$, and that for some $x_0
  \in \mathbb{C}^2$ and $r_0 > 0$ we have
  \[ \Theta(N_\kappa, x_0, r_0) > 2 - \kappa_0. \]
  Consider the rescaled Lagrangian $L = r_0^{-1}(N_\kappa -
  x_0)$. This has the following properties.
  \begin{itemize}
      \item[(1)] For
  any $R > \epsilon$ we have that $R^{-1}L$ is an
  $\epsilon$-graph over $P_1\cup P_2$ on the annulus $A_{1,2}$.
   \item[(2)] $L$ is Hamiltonian isotopic to $P_1\# P_2$ in  $B_1$  with an isotopy
      that is within $\epsilon$ of the identity near $\partial B_1$. 
  \item[(3)] Identifying the two components of $L\setminus
    B_\epsilon(0)$
    with $P_1, P_2$, we have $|\theta| < \kappa\epsilon$
    on the $P_1$-component, and $|\theta - \kappa| < \kappa\epsilon$
    on the $P_2$-component.
    \end{itemize}
\end{lemma}
In other words, if $\kappa$ is sufficiently small, then at scales
where the Gaussian area of $N_\kappa$ is close to 2, the Lagrangian
$N_\kappa$ is very close to the union of planes $P_1\cup P_2$, but
with the Lagrangian angle on the two planes differing by $\kappa$. 
\begin{proof}
  We argue by contradiction. Suppose that we have a sequence $\kappa_i
  \to 0$ and corresponding $x_i, r_i$ such that $\Theta(N_{\kappa_i},
  x_i, r_i) \to 2$. We examine the possible (subsequential)
  limits of the sequence
  $r_i^{-1}(N_{\kappa_i} - x_i)$, and claim that the only possibility
  that has Gaussian area 2 is the union $P_1\cup P_2$. Note first that
  by the monotonicity formula it follows that $\Theta(N, x, r) < 2$
  for all $x, r$, using that the tangent cone at infinity of $N$ has density
  2.

  If the $r_i$ remain uniformly bounded, then any limit of 
  $r_i^{-1}(N_{\kappa_i} - x_i)$ is either a plane (if $r_i^{-1}x_i\to
  \infty$), or a translate and
  scaling of $N$. These contradict $\Theta(N_{\kappa_i},
  x_i, r_i) \to 2$. We can therefore assume that $r_i\to\infty$. 

  We can also assume that $r_i^{-1}x_i\to 0$. Otherwise
  we would have $r_i^{-1} |x_i| \geq c_1 > 0$, in which case the sequence
  $r_i^{-1}(N_{\kappa_i} - x_i)$ would converge on compact sets to one
  of three possibilities: the empty set, a plane, or $(P_1\cup P_2)- x$ for some
  $|x| \geq c_1$. These would also contradict $\Theta(N_{\kappa_i},
  x_i, r_i) \to 2$.

  It follows therefore that $r_i \to \infty$ and $r_i^{-1}x_i \to 0$,
  in which case $r_i^{-1}(N_{\kappa_i} - x_i)$ converges to $P_1\cup
  P_2$. More precisely, by Lemma~\ref{lem:Nkappa1}, $r_i^{-1}N_{\kappa_i}$ is the graph of
  $\tilde{v}$ over $r_i^{-1}N$, where $\tilde{v}$ satisfies the same
  estimates as $v$. This, together with
  $r_i^{-1}x_i\to 0$ shows the claim (1). The claim
  (2) follows by noting that $N$ is Hamiltonian isotopic to $\lambda
  N$ for any $\lambda$, with the isotopy converging to the identity at
  infinity.
  The claim (3) follows from the last property in
  Lemma~\ref{lem:Nkappa1}. 
\end{proof}

We will work in a compact Calabi-Yau surface $(X,\omega, J, \Omega)$,
and the Lagrangian initial condition that we consider will only equal
$N_\kappa$ inside a Darboux chart. We make the following definitions
in order to keep track of how well this setup approximates the model
above.

\begin{definition}\label{defn:Darboux}
  Given $p\in X$ and $R_0 > 0$, let $B_{4R_0}$ denote the Euclidean
  $4R_0$-ball in $\mathbb{C}^2$. We say that the map $\phi:B_{4R_0}
  \to X$ is an $R_0$-Darboux chart at $p$ if the following conditions
  hold:
  \begin{enumerate}
    \item $\phi(0)=p$, and $\phi^*\omega$ agrees with the standard
      symplectic form on $\mathbb{C}^2$.
    \item The pullbacks of the complex structure and holomorphic
      volume form of $X$ satisfy
      \[  \Vert \phi^*J - J_0 \Vert_{C^3} + \Vert \phi^*\Omega -
  \Omega_0\Vert_{C^3} < R_0^{-1}, \]
where $J_0, \Omega_0$ are the standard structures on $\mathbb{C}^2$.
\item We have an isometric embedding $G: (X,g)\to\mathbb{R}^N$ such
  that $G\circ \phi$ is $R_0^{-1}$-close in $C^3$ to the map $x\mapsto
  (x,0)$, and the second fundamental form of $G(M)$ is bounded by
  $R_0^{-1}$.
  \end{enumerate}
\end{definition}

Note that for any $p\in X$ and $R_0 > 0$ we can find an $R_0$-Darboux
chart at $p$, as long as we first replace the metric $g$ on $X$ by
$R^2g$ for sufficiently large $R$. Below we will always implicitly
assume that we have scaled the metric on $X$ up in order to obtain
suitable Darboux charts.

\begin{definition}\label{defn:model2}
  Let $K_0, R_0, \kappa_0, \epsilon_0 >
  0$. We say that the Lagrangian $L_0\subset X$ is $K_0$-bounded, and
  equal to $\epsilon_0N_{\kappa_0}$ in an $R_0$-Darboux chart, if the
  following conditions hold:
  \begin{enumerate}
  \item We have an $R_0$-Darboux chart $\phi:B_{4R_0}\to X$, and
    corresponding isometric embedding $G:X\to \mathbb{R}^N$.
  \item We have area bounds $|L_0\cap B(x, r)| \leq K_0r^2$ for all $r
    > 0$ and     $x\in \mathbb{R}^N$ (the balls are Euclidean balls in
    $\mathbb{R}^N$), and Lagrangian angle bound $|\theta| < K_0$.
  \item In the Darboux chart we have $\lambda|_{\phi^{-1}(L_0)} =
    d\beta$, where $|\beta| \leq K_0(1 + |x|^2)$.
  \item We have $\phi^{-1}(L_0) = \epsilon_0N_{\kappa_0}$. 
  \end{enumerate}
\end{definition}

Throughout the paper we think of $K_0$ as a fixed large constant. The
parameter $\kappa_0$ determines how much we widen the angle between
the two asymptotic planes of the Lawlor neck, and it will be taken to
be very small. This is in contrast with Neves's construction in
\cite{Neves13}, where $\kappa_0$ only needs to be in
$(0,\pi/2)$. The parameter $\epsilon_0$ will also be chosen very
small, to ensure that the initial Lagrangian is very close to two
planes. One point to keep in mind is that the smaller $\kappa_0$ is,
the closer $N_{\kappa_0}$ is to being special Lagrangian. As
such, if we are hoping for a singularity to form at a time $t < 2$
with initial Lagrangian $\epsilon_0 N_{\kappa_0}$, then
we will need to choose $\epsilon_0$ correspondingly small too. 
Finally $R_0$ determines how well our Lagrangian in $X$ can be
approximated by the model Lagrangian in a large ball in
$\mathbb{C}^2$. 

\section{Existence of singularities}\label{sec:singexist}
In this section we show that if a Lagrangian $L \subset X$ is sufficiently close
to the model $N_\kappa$ inside a large ball, then a neck pinch singularity
must form. More precisely, we will always assume the following:
\begin{itemize}
  \item[(\dag)] The initial Lagrangian $L_0$ is $K_0$-bounded, and equals $\epsilon_0N_{\kappa_0}$ in
    an $R_0$-Darboux chart as in Definition~\ref{defn:model2}.
  \end{itemize}
We will think of $K_0$ as fixed, choose $\kappa_0$ very small, and
finally choose $\epsilon_0, R_0^{-1}$ small depending on $K_0,
\kappa_0$. 

Consider
the Lagrangian mean curvature flow $L_t$ with initial condition
$L_0$. 
We first have the following result, relying on the ideas used by
Neves~\cite{Neves13}.
\begin{prop}\label{prop:Neves}
  There exists $\kappa_0$ sufficiently small (depending
  on $K_0$) and $R_0$ large, $\epsilon_0$ small, depending on $K_0, \kappa_0$, with the
  following property. There is a time $T_0\in (0,2)$ such that one
  of the following holds:
  \begin{itemize}
    \item[(a)] either the flow starting with $L_0$ develops a singularity
      outside of the ball $B_{R_0}$ at time $T_0$,
    \item[(b)] or there is a point $(x_0, T_0)$ with $x_0\in B_{R_0}$ where the
      flow has density at least 2.
    \end{itemize}
  \end{prop}
\begin{proof}
  This follows essentially from the arguments in
  \cite{Neves13}. We briefly recall the main points. Let us assume
  that the flow exists smoothly for $t\in [0,T)$. If $T < 2$, and
  the flow cannot be extended further, then a
  singularity must form somewhere, and since the flow has density at
  least 2 at any singularity, it follows that either case (a) or (b) must
  hold.

  Suppose therefore that $T \geq 2$. We will apply
  Neves~\cite[Theorem 4.1]{Neves13}. In the setting of Neves's
  work we should think of the planes $P_1, P_2$ and $\kappa_0$ as
  fixed (note also that what he calls the plane $P_2$ is $P_{2,\kappa}$ in our
  notation in Lemma~\ref{lem:Nkappa1}). For the Lagrangian planes $P_1,
  P_{2,\kappa}$ as in Lemma~\ref{lem:Nkappa1}, for small $\kappa$,
  there is a unique expander 
  $\mathcal{S}\subset\mathbb{C}^2$ asymptotic to $P_1$ and $P_{2,\kappa}$ (see
  Joyce-Lee-Tsui~\cite{JLT10} for the existence and
  Lotay-Neves~\cite{LN13} for the uniqueness). This expander is
  Hamiltonian isotopic to the connected sum $P_{2,\kappa} \# P_1$, where the
  isotopy converges to the identity at infinity.

  We now apply Neves's result~\cite[Theorem 4.1]{Neves13}. Let us
  fix a large $S_0 > 0$ and small $\nu > 0$. The flow $Q_{1,t}$ in
  Neves's work is the same as $L_t$ in our notation. The result
  implies that if for a given $\kappa_0$ we choose $R_0$
  sufficiently large, and $\epsilon_0$ sufficiently small (depending
  on $\kappa_0, K_0, \nu, S_0$), then $L_1$
  is $\nu$-close to the expander $\mathcal{S}$ in the ball $B_{S_0}$
 If $S_0$ is chosen sufficiently large, then
  inside $B_{S_0}$ the expander $\mathcal{S}$ is Hamiltonian isotopic
  to $P_2\# P_1$, where the isotopy is close to the identity near
  $\partial B_{S_0}$. It follows from Theorem~\ref{thm:Seidel}
  that the flow cannot be embedded for all times $t\in [0,1]$, since
  the initial Lagrangian $L_0$ is Hamiltonian isotopic to $P_1\#
  P_2$ inside the ball $B_{S_0}$, again with an isotopy close to the
  identity near $\partial B_{S_0}$.
  Therefore there must be an immersed point at which the density of
  the flow is at least 2. 
\end{proof}

The main new ingredient is the following.

\begin{prop}\label{prop:singexist}
  Given $K_0 > 0$, we can choose $\kappa_0$ sufficiently small, and
  then $\epsilon_0, R_0^{-1}$ small, depending on
  $K_0, \kappa_0$,  with the
  following property. Suppose that $L_0$ is as in (\dag)
  above. Then at some time $T_0 < 2$ the flow with initial condition $L_0$
  develops a singularity at a point $x_0$. If $x_0\in B_{R_0}$, then the
  singularity is a neck pinch, i.e. the tangent flow is given by the
  union of two transverse planes. 
\end{prop}
\begin{proof}
  Let $T_0$ denote the infimum of times at which the conclusion of
  Proposition~\ref{prop:Neves} holds. Then at time $T_0$ the flow either develops a
  singularity at a point $x_0\not\in B_{R_0}$, or the flow has density
  at least $2$ at $(x_0,T_0)$, with $x_0\in B_{R_0}$. We claim that in the
  latter case the flow has a neck pinch singularity at $(x_0,T_0)$, using
  the results of \cite{LSSz22}, as long as the parameters $\kappa_0$
  is sufficiently small, and $R_0$ is sufficiently large (depending on
  $\kappa_0$). The parameter $\epsilon_0$ is determined by
  Proposition~\ref{prop:Neves}.
  Note that for times $t < T_0$ the flow is embedded in $B_{R_0}$. 

  Let us consider the rescaled flow $M_\tau$ centered at $(x_0,T_0)$, translated
  in time so that the initial condition $M_0$ is a suitable scaling of
  $L_0-x_0$. That is, we
  consider $M_\tau$ defined by
  \[ \label{eq:rescaled1} M_\tau = e^{(\tau_0+\tau)/2}(L_{T_0-e^{-\tau_0-\tau}}-x_0), \]
  where $\tau_0$ is chosen so that $T_0 = e^{-\tau_0}$. We can then write
  \[ M_0 = r_0^{-1}(L_0 - x_0), \]
  where $r_0 = \sqrt{T_0}$. 

  Since the
  density of the flow at $(x_0, T_0)$ is at least 2, we have
  \[ \lim_{\tau \to \infty} \frac{1}{4\pi} \int_{M_\tau} e^{-|x|^2 /
      4}\, d\mathcal{H}^2 \geq 2. \]

  Given any $\epsilon_1 > 0$, if $R_0$ is sufficiently large (so that the
  second fundamental form of $X\subset \mathbb{R}^N$ is sufficiently small), then
  from the monotonicity formula we have that
  \[ \label{eq:L0area}\frac{1}{4\pi} \int_{M_0}  e^{-|x|^2 / 4}\,
    d\mathcal{H}^2 \geq 2 - \epsilon_1. \]
  In terms of the initial Lagrangian $L_0$ this means that
  \[ \Theta(L_0, x_0, r_0) \geq 2-\epsilon_1, \]
  and note that $r_0 \leq \sqrt{2}$, while $|x_0| \leq R_0$. Using the
  identification given by our Darboux chart, we have that $L_0 \cap
  B_{4R_0} = \epsilon_0 N_{\kappa_0}$. If $R_0$ is chosen sufficiently
  large (depending also on the area bounds given by $K_0$), we have
  \[ \Theta(\epsilon_0 N_{\kappa_0}, x_0, r_0) \geq 2 - 2\epsilon_1. \]
  By this discussion, and using Lemma~\ref{lem:NkappaG}, for any
  $\epsilon_2 > 0$ we can choose 
  $\kappa_0$ sufficiently small, and then $R_0$ sufficiently large so
  that $r_0^{-1} (\epsilon_0N_{\kappa_0} - x_0)$ is
  $\epsilon_2$-graphical over
  $P_1\cup P_2$ over the annulus $A_{\epsilon_2, \epsilon_2^{-1}}$.
  Choosing $R_0$ even larger, it follows that
  $\tilde{L}_0 = r_0^{-1}(L_0 - x_0)$ has the same graphicality property over
  $P_1\cup P_2$ on $A_{\epsilon_2, \epsilon_2^{-1}}$.

  We can now apply Proposition~\ref{prop:LSSzunique}. The conclusion is
  that given any $\delta_2 > 0$, if $\epsilon_2$ above and
  $\epsilon_1$ in \eqref{eq:L0area} are sufficiently small, then for
  all $\tau > 0$ the Lagrangian $M_\tau$ is a $\delta_2$-graph
  over $P_1\cup P_2$ on the annulus $B_2\setminus B_1$. In particular
  any subsequential limit $\lim_{\tau\to\infty} M_\tau$ is also a $\delta_2$-graph
  over $P_1\cup P_2$. These limits are the possible tangent flows of
  the flow $L_t$ at $(x_0, T)$, and so if $\delta_2$ is sufficiently
  small, the tangent flow is necessarily a transverse union of two
  planes. It follows that a neck pinch singularity must form at $(x_0,
  T)$, since if no singularity were to form, then the flow would
  already be an immersed flow for times $t < T$ close to $T$. 
\end{proof}

Using this result we can prove Theorem~\ref{thm:2} from the
introduction.
\begin{proof}[Proof of Theorem~\ref{thm:2}]
  We take $(M, g)$ to be an elliptically fibered $K3$-surface, with a
  singular fiber $C_0$ that has a node at $p$. In particular there is a
  holomorphic map $\pi:M \to \mathbb{P}^1$, with respect to a
  complex structure $I$, such that in a neighborhood of $p$, in
  suitable holomorphic coordinates $w_1,w_2$ we have $\pi(w_1,w_2) =
  w_1w_2$. The curve $C_0$ 
  is given by the fiber $\pi^{-1}(0)$ in these coordinates, and we can
  assume that $C_0$ is smooth otherwise. Let us write $C_t =
  \pi^{-1}(t)$ for small $t\in \mathbb{C}$. It follows that for
  sufficiently small $t\not=0$ the curve $C_t$ is a smooth genus one
  curve, for the complex structure $I$.
  Let $J$ be a suitable hyperk\"ahler rotated complex
  structure, so that the $C_t$ are special Lagrangian tori, converging
  to a singular special Lagrangian $C_0$ as $t\to 0$. 

  We will take the initial Lagrangian $L_0$ of the flow to be a small
  Hamiltonian perturbation of  $C_t$,
  for suitable sufficiently small $t$ and we need to ensure that $L_0$
  satisfies the hypotheses of Proposition~\ref{prop:singexist}, after
  sufficiently scaling up the metric $g$.
  We can choose holomorphic coordinates $z_k=x_k + iy_k$ 
  for $k=1,2$ centered at $p$ (with respect to $J$)
  such that the metric $g$ and the holomorphic volume form at $p$
  agree with the standard structures on $\mathbb{C}^2$. The two curves
  $w_k^{-1}(0)$ define two Lagrangian planes $P_1, P_2$ in $T_pM =
  \mathbb{C}^2$. Without loss of generality we can assume that $P_1$
  is the $x_1x_2$-plane, and
  \[ P_2=\{(e^{i\theta_1}x_1,  e^{i\theta_2}x_2)\,:\,
    x_1,x_2\in\mathbb{R}\}, \]
  where $\theta_k\in(0,\pi)$ and $\theta_1 +\theta_2 = \pi$. In other
  words $P_1$ is given by the equation $y_1 + iy_2=0$, while $P_2$ is
  defined by $x_1 - ix_2 = b(y_1 + iy_2)$, where $b=\cot\theta_1$. It follows that
  to leading order we can assume that the holomorphic coordinates
  $w_k$ for the complex structure $I$ above are given to leading order
  by
  \[ W_1 = y_1 + iy_2, \text{  and } W_2 = x_1 - ix_2 - b(y_1+iy_2).\]

  If we consider rescalings $R^2g$ of the metric $g$, and
  corresponding scalings $Rz_k, Rw_k$ of our coordinates, then as
  $R\to \infty$, the curves $C_t$ converge to the Lawlor necks
  $W_1W_2=t'$ inside $\mathbb{C}^2$. More precisely, for any $R_0, \nu
  > 0$ and $|t'| = 1$, we can find a suitable scaling $R^2g$ and
  $R_0$-Darboux chart $\phi:B_{4R_0}\to M$ centered at $p$, such that
  for some $t$, the curve $\phi^{-1}(C_t)$ is $\nu$-close in $C^3$ to the Lawlor neck
  $W_1W_2=t'$ on $B_{4R_0}$. Let us write $\tilde{N}_{t'} =
  \phi^{-1}(C_t)$ for this Lagrangian perturbation of the Lawlor
  neck. We claim that for suitable values of $t'$, $\tilde{N}_{t'}$ is
  exact in the ball $B_{4R_0}$. Note that the Lawlor neck given by
  $W_1W_2=t'$ contains a loop $\rho_{t'}$ parametrized by $W_1 = t'e^{i\theta},
  W_2 = e^{-i\theta}$, and if $\nu$ is sufficiently small, then this
  circle can be projected to a loop $\gamma_{t'}$ in
  $\tilde{N}_{t'}$. It is enough to check that for suitable $t'$ we
  have $\int_{\gamma_{t'}} \lambda = 0$ for the Liouville form
  $\lambda$. 

  We first consider the Lawlor necks given by $W_1W_2 = t'$, and the
  loop $\rho_{t'}$. Up to adding an exact form, the Liouville form
  $\lambda$ is $-2(y_1dx_1 + y_2dx_2)$. Therefore it is enough to
  consider
  \[ \int_{\rho_{t'}} y_1\,dx_1 + y_2\,dx_2. \]
  Along $\rho_{t'}$ we have
  \[ y_1 + iy_2 &= W_1 = t' e^{i\theta}, \\
    x_1 - ix_2 &= W_2 + bW_1 = e^{-i\theta} + bt' e^{i\theta}, \]
  therefore
  \[ y_1\,dx_1 + y_2\,dx_2 &= \mathrm{Re}\, (y_1+iy_2)\,d(x_1-ix_2) \\
    &= \mathrm{Re}\, (-it' + ibt'^2e^{2i\theta}) d\theta. \]
  When $t' = i$ we have
  \[ \int_{\rho_i} y_1\,dx_1 + y_2\,dx_2 = 2\pi, \]
  while at $t'=-i$ we have
  \[  \int_{\rho_{-i}} y_1\,dx_1 + y_2\,dx_2 = -2\pi. \]
  It follows that once $\nu$ above is sufficiently small, we have
  $\int_{\gamma_i} \lambda < 0$, while $\int_{\gamma_{-i}} \lambda >
  0$. Therefore by continuity there will be at least two values of
  $t'$ for which the Lagrangian $\tilde{N}_{t'}$ is exact inside the
  ball $B_{4R_0}$.

  If we take $R_0$ sufficiently large, and $\nu$
  sufficiently small, then we can take a small Hamiltonian deformation
  of $L_0$ of the corresponding $C_t$ to ensure that after a further
  scaling by $\epsilon_0$, $L_0$ satisfies
  the conditions in Definition~\ref{defn:model2}. By
  Proposition~\ref{prop:singexist} we can ensure that the flow with
  initial condition $L_0$ has a singularity at time $T < 2$. 
\end{proof}

\section{Nondegenerate neck pinches} \label{sec:nondeg}
In this section we study further properties of the singularities that
appear in Proposition~\ref{prop:singexist}, and introduce the notion
of nondegenerate neck pinches in Definition~\ref{defn:nondeg}. 

Suppose that
we are in the setting of Proposition~\ref{prop:singexist}. Thus, $L_t$ is a graded,
rational Lagrangian mean curvature flow in a compact Calabi-Yau
surface $X$, which develops a neck pinch singularity at $(x_0,
T_0)$. This means that the tangent flow at the
singularity is the special Lagrangian union $P_1' \cup P_2' $ of two
transverse planes. Note that these two planes will in general be small
perturbations of the planes $P_1\cup P_2$ used in the construction of
the Lawlor neck $N$ above. 

As in \eqref{eq:rescaled1},
we consider the rescaled flow $M_\tau$ centered at $(x_0, T_0)$,
translated in time so that $M_0$ is a rescaling of $L_0 - x_0$. From
Lemma~\ref{lem:NkappaG} we know that given any
$\epsilon_2 > 0$, we can choose $\kappa_0$ sufficiently small, and then 
$R_0$ sufficiently large (depending on $\epsilon_2, \kappa_0, K_0$) so that we
have the following:
\begin{itemize}
  \item[(a)] On the annulus $A_{\epsilon_2, R_0}$, $M_0$ is $\epsilon_2$-graphical over
    $P_1\cup P_2$. 
  \item[(b)] In the annulus $B_{R_0} \setminus B_{\epsilon_2}$ the
    Lagrangian angle on $M_0$ satisfies $|\theta|  <
    \epsilon_2$ on the component corresponding to $P_1$, and $|\theta
    - \kappa_0| < \epsilon_2\kappa_0$ on the component corresponding
    to $P_2$.
    \item[(c)] In $B_1$, $M_0$ is Hamiltonian isotopic to $P_1\# P_2$,
      with an isotopy that can be made arbitrarily close to the identity near $\partial
      B_1$ by taking $\epsilon_2$ sufficiently small. 
  \end{itemize}
Using the fact that the limit as $\tau\to \infty$ of the rescaled flow
$M_\tau$ is $P_1' \cup P_2'$, it follows from
Proposition~\ref{prop:LSSzunique} that $M_\tau$ remains close 
to the two planes $P_1\cup P_2$ for all time. More precisely, we have
the following.

\begin{prop}\label{prop:goodinitialcond}
  Let $\epsilon_3 > 0$. We can choose $\kappa_0$ sufficiently small,
  depending on $\epsilon_3$, and then $\epsilon_0, R_0^{-1}$ sufficiently small,
  depending on $\epsilon_3, \kappa_0, K_0$, such that we have the
  following. Suppose that $L_0$ is a Lagrangian as in
  Proposition~\ref{prop:singexist}, equal to $\epsilon_0 N_{\kappa_0}$
  in an $R_0$-Darboux chart, so that the flow with initial
  condition $L_0$ develops a neck pinch singularity at $(x_0, T_0)$
  for some $x_0 \in B_{R_0}$ and $T_0 < 2$. Let $M_\tau$ be
  the rescaled flow centered at $(x_0, T_0)$, so that $M_0$ is
  a scaling of $L_0 - x_0$.

  Then on the annulus $A_{\epsilon_3, \epsilon_3^{-1}}$ the
surfaces $M_\tau$ are $\epsilon_3$-graphical over $P_1\cup
P_2$ for all $\tau \geq 0$, and in addition we have
\[ \label{eq:L4bound}
  \left(\int_{M_0} \theta^4\, e^{-|x|^2/4} \, d\mathcal{H}^2\right)^{1/4} <
    C\kappa_0, \]
  for a uniform $C$ (depending only on the Gaussian area of the
  plane). Note that instead of the $L^4$-norm of $\theta$ we could
  take any other $L^p$ norm here. 
\end{prop}

\begin{proof}
  The first statement follows from Proposition~\ref{prop:LSSzunique}. For
  \eqref{eq:L4bound} we use the property (b) above. On the one hand we
  have
  \[ \int_{M_0 \cap B_{R_0}} \theta^4\, e^{-|x|^2/4}\,
    d\mathcal{H}^2 < C \kappa_0^4, \]
  while by choosing $R_0$ large enough, depending on $\kappa_0, K_0$
  we can arrange that
  \[ \int_{M_0 \setminus B_{R_0}} \theta^4\, e^{-|x|^2/4}\,
    d\mathcal{H}^2 < C\kappa_0^4. \]
\end{proof}

Recall that
the Lagrangian angle $\theta$ satisfies the drift heat equation
\[ \partial_\tau \theta = \Delta\theta - \frac{1}{2} x\cdot
  \nabla\theta \]
along the rescaled flow. Note that we are viewing $X\subset
\mathbb{R}^N$ as being isometrically embedded, so that $M_\tau
\subset e^{(\tau+c)/2}X$, where the choice of $c$ is given by
our time translation of the rescaled flow. Moreover the rescaled flow 
$M_\tau$ evolved in a normal direction,
as immersions satisfying
\[ \label{eq:RMCFeqn} \frac{\partial x}{\partial \tau} = \mathbf{H} + \frac{x^\perp}{2}
  + \nu,\]
for a forcing term $\nu$ given in terms of the second
fundamental form of $e^{(\tau + c)/2}X$. Choosing $R_0$ large, we have $|\nu| \leq
e^{-(\tau+c)/2}$, and so from the monotonicity formula and \eqref{eq:L4bound}
we have
\[ \int_{M_\tau} \theta^4\, e^{-|x|^2/4} \, d\mathcal{H}^2 <
  2C\kappa_0^2, \]
for a uniform $C > 0$ for all $\tau > 0$.

Using Ecker's log Sobolev inequality~\cite{Ecker} we have the
following more refined estimates for subsolutions of the drift heat equation,
similar to \cite[Lemma 3.5]{LSSz22}. 

\begin{prop}\label{prop:Ecker}
  Suppose that $u(x,\tau) \geq 0$ is defined along the flow $M_\tau$,
  has polynomial growth, and satisfies
  \[ \label{eq:subsol1} \frac{\partial u}{\partial \tau} \leq \Delta u - \frac{1}{2}
   x\cdot \nabla u. \]
  Suppose that $R_0$ is chosen sufficiently large, as above, so that
  $|\nu|\leq e^{-\tau/2}$ along the flow. We have the following.
  \begin{itemize}
  \item[(a)] Letting $p(\tau) = 1 + e^\tau$, we have
    \[ \left( \int_{M_\tau} u^{p(\tau)}\, e^{-|x|^2/4}\,
        d\mathcal{H}^2 \right)^{1/p(\tau)} \leq C \left(
        \int_{M_0} u^2\, e^{-|x|^2/4}\, 
        d\mathcal{H}^2 \right)^{1/2}, \]
    where $C > 0$ is a universal constant.
  \item[(b)] For $0 < \delta \leq s \leq 1$, there are $p > 1$ and $C
    > 0$, depending on $\delta$, such that on $M_s$ we have
    \[ u(x, s)^2 \leq C
      e^{\frac{|x|^2}{4p}}\int_{M_0} u(x,
      0)^2\, e^{-|x|^2/4}\, d\mathcal{H}^2. \]
  \end{itemize}
\end{prop}

Using that $M_\tau$ converges to the union $P_1'\cup P_2'$ as
$\tau\to \infty$, together with the fact that by the monotonicity
formula $\theta$ remains uniformly bounded, we can extract a suitable
normalized limiting solution of the drift heat equation on $P_1'\cup
P_2'$. More precisely, we consider a sequence $\tau_i\to \infty$, and
let $L^i_\tau = M_{\tau-\tau_i}$ be the corresponding
translated flows. Let $\underline\theta_i$ be the average of $\theta$ on
$L^i_0$ with respect to the Gaussian area, i.e. 
\[ \int_{L^i_0} (\theta - \underline\theta_i)\, e^{-|x|^2/4}\,
  d\mathcal{H}^2 =0, \]
and let $\theta_i$ denote the normalized angle function
\[ \label{eq:tnorm1} \theta_i = \frac{ \theta - \underline\theta_i }{ \left(\int_{L^i_0}
      (\theta-\underline\theta_i)^2\, e^{-|x|^2/4}\,
      d\mathcal{H}^2\right)^{1/2}}. \]
These satisfy the drift heat equation along $L^i_\tau$, for $\tau\in
[-1,2]$, say. Using the monotonicity formula, we have uniform bounds
for the $\theta_i$ on any compact subset of $L^i_\tau$ for $\tau >
0$. It follows that as $i\to\infty$, we can extract a limit solution $\theta_\infty$ of
the drift heat equation on $P_1'\cup P_2'$, for $\tau\in (0,2]$. This
means that $\theta_\infty$ is given by a pair of solutions
$(\theta_{\infty, 1}, \theta_{\infty,2})$ on the planes $P_1', P_2'$. Note
that a priori these solutions could be identically zero. We define
nondegenerate neck pinches in terms of this limit solution.

\begin{definition}\label{defn:nondeg}
  Suppose that the flow $L_t$ develops a neck pinch singularity at
  $(x_0, T_0)$, with tangent flow $P_1\cup P_2$. Let $(\theta_{\infty,
    1}, \theta_{\infty, 2})$ denote the normalized limit of the
  Lagrangian angle on the two planes. We say that the neck pinch is
  \emph{nondegenerate} if
  \[ (\theta_{\infty, 1}, \theta_{\infty, 2}) = ( - (8\pi)^{-1/2},
    (8\pi)^{-1/2}), \]
  and in addition in a suitable Darboux chart near the
  singularity the Lagrangian $L_t$ for $t < 0$ sufficiently close to
  $0$ is Hamiltonian isotopic to $P_1 \# P_2$, with an isotopy that is
  close to the identity near the boundary of the chart.
\end{definition}

Let us recall that homogeneous solutions of the drift heat equation on
$\mathbb{R}^2$ correspond to eigenfunctions of the operator $-\Delta +
\frac{1}{2} x\cdot \nabla$. The eigenfunctions with eigenvalue $k/2$
are homogeneous polynomials of degree $k$ given by products of Hermite
polynomials (see Bogachev~\cite[Chapter 1]{Bog98}). The 
solution of the drift heat equation corresponding to the eigenfunction
$f(x)$ is $e^{-k\tau/2}f(x)$. In particular the only homogeneous solution that
does not decay is the constant. Because of this, it is natural to
expect that in the most nondegenerate settings the limit solution
$\theta_\infty$ constructed above is given by $(c, -c)$ on the two
planes for a suitable nonzero constant $c$, determined by the
$L^2$-normalization to be $\pm (8\pi)^{-1/2}$. This motivates the
definition of nondegenerate neck pinches above, with the added
requirement that in the connect sum decomposition $P_1\# P_2$ the
limiting solution is negative on the first summand.

Our next goal is to show that in the setting of
Proposition~\ref{prop:singexist}, if the parameters $\kappa_0, R_0,
\epsilon_0$ are chosen suitably, then the neck pinch singularity that
appears is nondegenerate. For this we next recall a version of
the three annulus lemma.

For any function $u(x, \tau)$ defined along the rescaled flow $M_\tau$ we
define the norms
\[ \Vert u\Vert_\tau^2 = \int_{M_\tau} u(x, \tau)^2\,
  e^{-|x|^2/4}\, d\mathcal{H}^2. \]
In addition let us define the averages $\underline{u}_\tau$ of $u$ on
$M_\tau$.
We have the following, whose proof is very similar to
\cite[Proposition 3.9]{LSSzAncient}. 

\begin{lemma}\label{lem:3ann} 
  Let $K\subset \mathbb{R}\setminus \mathbb{Z}$ be compact. There is
  an $\epsilon > 0$, depending on $K$, such that if $s\in K$, then
  we have the following. Suppose that $M_\tau$ is a (forced) rescaled
  mean curvature flow satisfying \eqref{eq:RMCFeqn} for $\tau\in [0,2]$ with a forcing term $|\nu| <
  \epsilon$, and $M_\tau$ is $\epsilon$-graphical over
  $P_1\cup P_2$ on the annulus $A_{\epsilon, \epsilon^{-1}}$. Suppose
  in addition that $u$ solves the drift heat
  equation along $M_\tau$, and has polynomial growth. If  
  \[ \label{eq:3ann1} \Vert u - \underline{u}_1\Vert_{1} \geq e^{s/2}
    \Vert u - \underline{u}_{0}\Vert_{0}, \]
  then  
  \[ \label{eq:3ann2} \Vert u - \underline{u}_{2}\Vert_{2} \geq
    e^{s/2} \Vert u - \underline{u}_1\Vert_{1}. \] 
\end{lemma}
\begin{proof}
  Let us briefly sketch the proof. Suppose that we have a sequence of such flows
  $M^i_\tau$, converging to the union $P_1\cup P_2$ on compact
  subsets of $[0, 2] \times \mathbb{C}^2\setminus \{0\}$, and
  corresponding solutions $u^i$ satisfying \eqref{eq:3ann1} but not
  \eqref{eq:3ann2}, for numbers $s_i \in K$ in place of $s$.
  Without loss of generality we can assume that
  $\underline{u}^i_{0} = 0$, and $\Vert u^i -
  \underline{u}^i_{1}\Vert_1 = 1$. Then the hypotheses imply that
  \[ \Vert u^i\Vert_{0} \leq e^{-s_i/2}, \]
  and
  \[ \Vert u^i - \underline{u}^i_{2}\Vert_{2} \leq e^{s_i/2}. \]
  Up to choosing a subsequence we can assume that the $u^i$ converge
  smoothly on compact subsets of $(0, 2] \times
  \mathbb{C}^2\setminus \{0\}$ to a solution $u$ of the drift heat
  equation on $P_1\cup P_2$, and $s_i\to s\not\in\mathbb{Z}$.
  From the monotonicity formula, together
  with  Proposition~\ref{prop:Ecker}, as in \cite[Lemma 3.8]{LSSzAncient}, we find that
  \[ \Vert u - \underline{u}_1\Vert_1 = 1, \qquad \Vert u -
    \underline{u}_{2}\Vert_{2} \leq e^{s/2}, \]
  as well as
  \[ \lim_{\tau\to 0} \Vert u\Vert_{\tau} \leq e^{-s/2}. \]
  On $P_1\cup P_2$ the average of the solution $u_\tau$ is independent of
  $\tau$, so $\underline{u}_1 = \underline{u}_{2}$, and in
  particular the solution $v$ defined by $v = u - \underline{u}_1$ is
  $L^2$ orthogonal to the constants at all times, and
  satisfies
  \[ \Vert v\Vert_1 = 1, \qquad \Vert v\Vert_{2} \leq e^{s/2}, \]
  as well as
  \[ \lim_{\tau \to 0} \Vert v \Vert_\tau \leq e^{-s/2}. \]
  This implies that $v$ is a homogeneous solution of degree $s/2$,
  which is a contradiction if $s\not\in \mathbb{Z}$. 
\end{proof}

Let us return to the setting of
Proposition~\ref{prop:goodinitialcond}. 
We show that for suitable initial Lagrangians $M_0$,
the normalized angle function can only decay slowly. 

\begin{prop}\label{prop:slowdecay}
   Let $s > 0$. If $\kappa_0$ is sufficiently small, depending
   on $s$, and $\epsilon_0, R_0^{-1}$ are small, depending on $s,
   \kappa_0$, then 
   along the rescaled flow $M_\tau$ as above, we have
   \[ \Vert \theta - \underline{\theta}_1 \Vert_1 \geq e^{-s} \Vert
     \theta -\underline{\theta}_0 \Vert_0. \]
   In addition on the two components of $M_\tau\cap
   (B_2\setminus B_1)$ corresponding to $P_1$, $P_2$, for $\tau\in
   [0,1]$, we have
   $|\theta| < s \kappa_0$ on the $P_1$-component and $|\theta -
   \kappa_0| < s \kappa_0$ on the $P_2$-component. 
 \end{prop}
 \begin{proof}
   To show this, we consider sequences $\kappa_{0,i}\to 0$, $R_{0,i}\to
   \infty$, $\epsilon_{0,i}\to 0$, such that the corresponding
   rescaled flows $M^{i}_{ \tau}$ constructed as above, satisfy
   the following:
   \begin{itemize}
     \item[(a)] On the annulus $B_{R_{0,i}} \setminus B_{1/i}$, for
       $\tau\in [0,1]$ the $M^i_{ \tau}$ are
       $i^{-2}$-graphical over $P_1\cup P_2$. Let us write
       $M^i_{\tau, 1}$, $M^i_{\tau,2}$ for the two
       corresponding components of $M^i_\tau$ inside this
       annulus. 
     \item[(b)] 
       On the component $M^i_{0,1}$ we have $|\theta| < i^{-1}
       \kappa_{0,i}$, while on $M^i_{0,2}$ we have $|\theta -
       \kappa_{0,i}| < i^{-1}\kappa_{0,i}$.
     \item[(c)]
       By choosing $R_{0,i}$ sufficiently large, we have the $L^4$ bound
       \[ \left( \int_{M^i_\tau} \theta^4\, e^{-|x|^2/4}\,
           d\mathcal{H}^2\right)^{1/4} < C\kappa_{0,i}, \]
       for $\tau=0$, which implies the same type of bound for all
       $\tau >0$ by the monotonicity formula.
     \item[(d)] The forcing term $\nu$ along the mean curvature flow
       satisfies $|\nu|^2  < i^{-1}$. 
       \end{itemize}

     The sequence of flows $M^i_\tau$ converges smoothly on
     compact subsets of $(\mathbb{C}^2 \setminus \{0\}) \times
     [0,\infty)$ to the static flow $P_1\cup P_2$ as $i\to
     \infty$. Let us write $\theta^{(i)}$ for the Lagrangian angle
     along $M^i_\tau$. By the uniform $L^4$ bound above, the normalized
     functions $\kappa_{0,i}^{-1} \theta^{(i)}$ converge smoothly on
         compact subsets of $(\mathbb{C}^2\setminus\{0\})\times
         (0,\infty)$ to a limit function $\theta^{(\infty)}$, which is
         a pair of solutions of the drift heat equation on the planes
         $P_1, P_2$.

         \bigskip
         \noindent {\bf Claim:} 
      We claim the following. given $\delta > 0$ and a compact set
      $K\subset (\mathbb{C}^2\setminus \{0\})$, there is a $\tau_0 >
      0$ such that for sufficiently large $i$ we have
      \[ \label{eq:theta30} |\theta^{(i)}| < \delta \kappa_{0,i}, \qquad \text{ for } \tau <
        \tau_0 \]
      on the component $M^i_{\tau, 1}$ in $K$.

      \bigskip
      To see this it
      is more convenient to consider the (unrescaled) mean curvature
      flows $L^i_t$ corresponding to $M^i_\tau$, normalized so
      that $L^i_{-1} = M^i_0$. Fix a large value of $i$, and
      consider the function $f = \max\{\kappa_{0,i}^{-1}\theta, 0\}$
      along $L^i_t$. This is a nonnegative subsolution of the heat equation, and
      so by the monotonicity formula \eqref{eq:mon} together with our
      bound for $|\nu|$, we have
      \[ \frac{d}{dt} \int_{L^i_t} f\, \rho_{x_1, t_1}\,
        d\mathcal{H}^2 \leq i^{-1} \int_{L^i_t} f\, \rho_{x_1,
          t_1}\, d\mathcal{H}^2, \]
      for any $x_1, t_1$. In particular if we choose $t_1\in (-1,0)$, then 
      \[ \label{eq:mineq1} f(x_1, t_1) \leq e^{i^{-1}} \int_{L^i_{-1}} f\, \rho_{x_1,
          t_1}\, d\mathcal{H}^2. \]
      Let us fix a small $c \in (0,1/10)$ such that if $x\in P_1$ satisfies
      $|x|=1$, then the ball $B_{2c}(x)$ is disjoint from $P_2$. 
      Let $x_1\not=0$. On $M^i_{-1}$, on the set where $|x-x_1| >
      c |x_1|$ we have
      \[\label{eq:rhoineq2}  \frac{\rho_{x_1,t_1}(x,-1)}{\rho_{0,0}(x,-1)} &= \frac{1}{ t_1+1}
        \exp\left(\frac{-|x-x_1|^2 + (t_1+1)|x|^2}{4(t_1+1)}\right) \\
        &\leq \frac{1}{t_1+1} \exp\left(-\frac{|x-x_1|^2}{8(t_1+1)}\right) \\
        &\leq \delta,\]
      as long as $t_1+1$ is sufficiently small, depending on $c,
      \delta$, and on a lower bound for $|x-x_1|$, or equivalently, for
      $|x_1|$. We view $c$ as fixed, since it only depends on the
      planes $P_1,P_2$. 
      If $K'\subset \mathbb{C}^2\setminus \{0\}$ is
      compact and $x_1 \in K'$ is contained in
      the component of $L^i_{t_1}$ corresponding to $P_1$ for
      sufficiently large $i$, then the
      closed ball $\overline{B_{c|x_1|}(x_1)}$ is still contained
      in $\mathbb{C}^2\setminus \{0\}$ and does not intersect the
      component of $L^i_{-1}$ corresponding to $P_2$.
      It follows from this, together with the property (b) above, that
      for sufficiently large $i$ we have
      \[ \int_{L^i_{-1} \cap B_{c|x_1|}(x_1)} f\,
        \rho_{x_1,t_1}\, d\mathcal{H}^2 \leq \delta. \]
      At the same time, if $t_1+1$ is sufficiently small (depending on
      $\delta, K'$, and the planes $P_1,P_2$), using \eqref{eq:rhoineq2} we have
      \[ \int_{L^i_{-1} \setminus B_{c|x_1|}(x_1)} f\rho_{x_1, t_1}\,
        d\mathcal{H}^2 \leq \delta \int_{L^i_{-1}\setminus B_{c|x_1|}(x_1)}
        f\rho_{0,0}\, d\mathcal{H}^2 \leq C\delta, \]
      using property (c) above as well. Combining these two estimates
      with \eqref{eq:mineq1} we get
      \[ f(x_1, t_1) \leq C\delta, \]
      for a larger $C$. 
      Applying the same argument with $f$ replaced by $-f$, we obtain
      the claim \eqref{eq:theta30}. Note that we can pass from
      estimates on
      $L^i_{t_0}$, with $t_0 > -1$ close to $-1$, to estimates on $M^i_\tau$
      with $\tau > 0$ close to zero, with a rescaling by a factor close
      to 1. 

      Repeating the argument with $f= \kappa_{0,i}^{-1}(\theta -
      \kappa_{0,i})$, we obtain an analogous estimate on the component
      $M^i_{\tau,2}$, corresponding to the plane $P_2$: given
      any $\delta > 0$ and a compact set $K\subset (\mathbb{C}^2
      \setminus\{0\})$, we have $\tau_0 > 0$ such that for all large
      $i$ we have
      \[ |\theta^{(i)} - \kappa_{0,i}| < \delta \kappa_{0,i}, \quad \text{
          for } \tau < \tau_0\]
      on the component $M^i_{\tau, 2}$ in $K$.

      Let us now pass the functions
      $\kappa_{0,i}^{-1}\theta^{(i)}$ to the limit as $i\to\infty$, as
      described above, to obtain a pair $(u_1, u_2)$ of solutions of
      the drift heat equation on the planes $P_1, P_2$. We claim that
      $u_1=0$ and $u_2 = 1$. To see that $u_1=0$, let us choose
      $\delta > 0$. The estimates above
      imply that for any compact set $K\subset P_1\setminus \{0\}$ we
      have some $\tau_0 > 0$ such that $|u_1(x, \tau)| < \delta$ for
      $0 < \tau < \tau_0$ and $x\in K$. Using the property (c) above,
      we also have the $L^4$ bound
      \[ \int_{P_1} |u_1|^4\, e^{-|x|^2/4}\, d\mathcal{H}^2 \leq C. \]
      We can therefore choose the set $K$ in a suitable way, depending
      on $C, \delta$, such that
      \[ \int_{P_1 \setminus K} |u_1|^2\, e^{-|x|^2/4}\,
        d\mathcal{H}^2 \leq \delta, \]
      for any $\tau > 0$. 
      It follows that for $\tau\in (0,\tau_0)$ we have
      \[ \int_{P_1} |u_1|^2\, e^{-|x|^2/4}\, d\mathcal{H}^2 \leq
        C\delta \]
      for a constant $C$ independent of $\delta$. By the monotonicity
      formula (on the static plane $P_1$), the same estimate follows
      for all $\tau > 0$. Letting $\delta\to 0$ we find that in fact
      $u_1=0$. The argument for $u_2$ is similar, working with the
      function $u_2-1$. 
 \end{proof}

 \begin{prop}\label{prop:nondeg10}
     Suppose that $\kappa_0$ is chosen sufficiently small (depending
     on $K_0$), and $\epsilon_0, R_0^{-1}$ are small, depending on
     $\kappa_0, K_0$. Then the neck pinch singularity that forms
     according to Proposition~\ref{prop:goodinitialcond} is
     nondegenerate. 
 \end{prop}

 \begin{proof}
 Let us consider the rescaled flow $M_\tau$ centered at
   $(x_0, T_0)$ as above, normalized so that $M_0$ is a
   suitable scaling of $L_0 - x_0$. Let $s > 0$ be small. By choosing the parameters
   $\kappa_0, R_0, \epsilon_0$ suitably, we can arrange that we are in
   the situation of Proposition~\ref{prop:slowdecay}, with this choice
   of $s$. In addition, using
   Proposition~\ref{prop:goodinitialcond}, we can assume that given
   any $\epsilon_3 > 0$,  the $M_\tau$ are $\epsilon_3$-graphical over
   $P_1\cup P_2$ in the annulus $A_{\epsilon_3, \epsilon_3^{-1}}$ for all
   $\tau > 0$. Since by property  (c) at the beginning of the
     section $M_0$ is Hamiltonian
   isotopic to $P_1\# P_2$ inside the ball $B_2$ (with the isotopy
   as close as we like to the identity near $\partial B_2$), it
   follows  that  $M_\tau$ for
   all $\tau > 0$ are also Hamiltonian isotopic to $P_1\# P_2$ in
   $B_2$. 

   Let us denote by $M^1_\tau, M^2_\tau$ the
   components of $M_\tau \cap A_{\epsilon_3, 2}$ corresponding
   to $P_1$ and $P_2$.    From the property (b) at the beginning of
   the section we also know
   that on $M^1_0$ we have $|\theta| < \kappa_0\epsilon_2$, while
   on $M^2_0$ we have $|\theta - \kappa_0| < \kappa_0\epsilon_2$. 
   In particular, assuming that $\epsilon_2$ is sufficiently small,
   it follows that $\theta$ is strictly larger
   on $M^2_0$ than on $M^1_0$. We will next show that the same
   holds for all $\tau > 0$ as well. 

   Choosing $\epsilon_3$ suitably, depending on $s$,
   Lemma~\ref{lem:3ann} implies that for all integers $k \geq 0$ we have
   \[ \Vert \theta - \underline{\theta}_{k+1} \Vert_{k+1} \geq e^{-s}
     \Vert \theta - \underline{\theta}_k\Vert_k. \]
   Up to scalar multiple, the only solution of the drift heat equation
   on the pair of planes $P_1, P_2$ which has zero average, and does
   not decay, is the pair of constant solutions $(1,-1)$. It follows
   that if $s, \epsilon_3$ are sufficiently small, then for all
   $k \geq 0$ the
   normalized angle function
   \[ \label{eq:ukdefn} u_k := \frac{ \theta - \underline{\theta}_k }{\Vert \theta -
       \underline{\theta}_k \Vert_k} \]
   satisfies
   \[ \label{eq:uksign}\left| u_k + \frac{1}{\sqrt{8\pi}}\right| &< 10^{-6}, \text{ on
     } M^1_\tau \\
      \left| u_k - \frac{1}{\sqrt{8\pi}}\right| &< 10^{-6}, \text{ on
      } M^2_\tau, \]
    for $\tau\in [k, k+10]$. The constant $\sqrt{8\pi}$ appears here
    because
    \[ \int_{P_1\cup P_2} e^{-|x|^2/4} = 8\pi. \]
    In addition, since the Lagrangian angle on $M^2_0$ is
    greater than on $M^2_0$, we can see inductively in $k$
    that the signs in \eqref{eq:uksign} can not switch.
 \end{proof}

 We now prove Theorem~\ref{thm:Joyceconj}.
 \begin{proof}[Proof of Theorem~\ref{thm:Joyceconj}]
    We follow the setting and notation of the proof of
    Proposition~\ref{prop:nondeg10}.  Recall that we have the rescaled flow $M_\tau$,
    converging to the union $P_1\cup P_2$ as $\tau\to\infty$, and in
    Proposition~\ref{prop:nondeg10} we studied the asymptotics of the
    Lagrangian angle along this rescaled flow. 
    In order to say something about the Lagrangian angle \emph{after} the
    singular time, we need to consider the original ``unrescaled'' flow
    $L_t$ instead. 

    We consider scalings (and translations) $L^k_t$ of $L_t$
    such that $L^k_{-1}$ is given by $M_k$. Specifically, we
    have
    \[ L^k_t =  e^{-\tau/2} M_{k+\tau},\]
    where $\tau = -\ln(-t)$. We can view this as a (forced) mean curvature flow in
    $\mathbb{R}^N$, and note that as $k\to\infty$, the flows $L^k_t$
    converge to the static flow given by the union $P_1\cup P_2$ of
    two planes.
    
    The flow $L^k_t$ exists smoothly for $t\in [-1,0)$, but as shown
    in \cite{LSSz22},  it can be
    extended to $t=0$ so that $L^k_0$ is a $C^1$ immersed
    submanifold. Moreover the flow can be continued smoothly for $t > 0$. For
    $k$ sufficiently large we can therefore assume that the flow
    $L^k_t$ exists for $t\in [-1,1]$, with a singularity at $(0,0)$,
    but smooth, and satisfying the (forced) mean curvature flow away
    from $(0,0)$. 

    As $k\to
    \infty$, the flows $L^k_t$ converge smoothly on compact subsets of
    $[-1,0)\times ( \mathbb{C}^2\setminus \{0\})$ to $P_1\cup P_2$,
    where $\mathbb{C}^2$ is identified 
    with the tangent space $T_{x_0}X$. It follows, using
    pseudolocality (see Ilmanen-Neves-Schulze~\cite[Theorem
    1.5]{INS19}), that the $L^k_t$ converge to
    $P_1\cup P_2$ smoothly on compact subsets of $[-1,1]\times (\mathbb{C}^2
    \setminus \{0\})$. Since for $t > 0$ the flow $M^k_t$ has two
    components inside $B_2(0)$, it follows that in fact we even have smooth
    convergence on compact sets of $(0,1] \times \mathbb{C}^2$
    (i.e. across $(0,1] \times \{0\}$) to $P_1\cup P_2$, where the two
    components converge to the two planes. 
    
    Let us consider the normalized Lagrangian angle function
    $U_k(x,t)$ along the flow $M^k_t$, where $U_k(x, -1)$ agrees with
    the function $u_k$ in \eqref{eq:ukdefn}. As $k\to\infty$, the
    functions $U_k$ converge to a solution of the heat
    equation on $P_1\cup P_2$, with the convergence being smooth on
    compact sets in $[-1,1]\times\mathbb{C}^2$ away from $[-1,0] \times \{0\}$. A priori this
    convergence is only along subsequences, but from the discussion
    above we know that this limit must be the pair of constants $(-(8\pi)^{-1/2},
    (8\pi)^{-1/2})$ on the two planes.

    In particular let us consider the space-time region
    \[ A := ([0, 1/2] \times B_{1/2}) \setminus ([0,1/10] \times
      B_{1/10}). \]
    The convergence $M^k_t\to (P_1\cup P_2)$ is smooth on this
    region, and in particular we can identify the two components
    $M^k_{t,1}, M^k_{t,2}$ converging to $P_1$ and $P_2$. 
    We know that for sufficiently large $k$, the function $U_k$ is
    strictly greater on $M^k_{t,2}$ than on $M^k_{t, 1}$ on the region
    $A$. After rescaling to the original flow, this means that the Lagrangian angle $\theta$
    is strictly greater on the component $L^2_{t}$ than on $L^1_{t}$
    in a spacetime region $A_k$. Here $L^1_{t}, L^2_{t}$ are the two
    components given locally after the singular time. Note, however
    that the union of these $A_k$ as $k\to \infty$ contains a
    spacetime region of the form $(T_0, T_0+\delta) \times
    B_\delta(x_0)$. In particular it follows that
    $\theta_{L_t^2}(p(t)) > \theta_{L^1_t}(p(t))$ for $t\in (T_0, T_0
    + \delta)$ as required.
\end{proof}

\section{Teardrop singularities}\label{sec:teardrop}
In this section we first recall some of the results from
\cite{LSSz22}, and then we will prove Theorem~\ref{thm:ndteardrop}.
We consider a rational, graded Lagrangian mean curvature flow $L_t\subset
\mathbb{C}^2$ for $t\in [0,T)$, with uniform area ratio bounds $|L_t \cap B_R(x)| \leq
CR^2$ for a constant $C > 0$. Note that it is enough to assume these
conditions for $L_0$, at the initial time (see \cite[Lemma
B.1.]{Neves07}).

We assume that at time $T_0$ the flow develops a teardrop singularity at
$\x_0$. This means that a tangent flow at the spacetime point
$(\x_0,T_0)$ is given by the union $P_1\cup_\ell P_2$, where $P_1$ and
$P_2$ are two Lagrangian planes intersecting in a line $\ell$. In
other words there is a sequence of scales $\sigma_i \to \infty$ such
that the parabolic rescalings
\[ L^i_t = \sigma_i (L_{T + \lambda_i^{-2} t} - \x_0)\]
converge weakly to the static flow $P_1\cup_{\ell} P_2$ on
$\mathbb{C}^2 \times [-1,0)$ as $i\to \infty$. For the flow in
$\mathbb{C}^2$ the rationality
assumption on the flow $L_t$ means that there is some $\alpha > 0$
such that for any loop $\gamma \subset L_0$ we have $\int_\gamma
\lambda \in \alpha \mathbb{Z}$ for the Liouville form
$\lambda$. Up to rescaling $L_t$ we can assume that $\alpha = 2\pi$,
and so we can define a function $\beta : L_t \to \mathbb{R} / 2\pi\mathbb{Z}$
satisfying $d\beta = \lambda|_{L_t}$. Similar primitives for $\lambda$
will also exist for the rescalings $L^i_t$, as long as the scaling
factors $\sigma_i$ are integers. 

Let us introduce real coordinates $x,y,z,w$ on $\mathbb{C}^2$ so that
$P_1 = \{x,w = 0\}$ and $P_2 = \{y,w = 0\}$. In order to define
\emph{nondegenerate} teardrop singularities, we study the asymptotics
of the function $w$ on the tangent flow, similarly to
Definition~\ref{defn:nondeg}, that used the asymptotics of $\theta$. To
define the normalized limit, we consider the rescaled flow
\[ M_\tau = e^{\tau/2} (L_{T_0 - e^{-\tau}} - \x_0). \]
By assumption along a sequence $\tau_i\to\infty$ we have $M_{\tau_i}
\rightharpoonup P_1\cup_\ell P_2$. As in \cite{LSSzAncient}, let us introduce the
rescaled coordinate functions $\tilde{x} = e^{-\tau/2} x$, and
analogously $\tilde{y}, \tilde{z}, \tilde{w}$, which are all solutions
of the drift heat equation
\[ \frac{\partial}{\partial \tau} f = \Delta f - \mathbf{x} \cdot
  \nabla f \]
along $M_\tau$. Let us define the 3-dimensional space $V =
\mathrm{span}\{\tilde{x}, \tilde{y}, \tilde{z}\}$ of solutions of the
drift heat equation along $M_\tau$. We are interested in extracting
the asymptotic behavior of $\tilde{w}$, ``modulo'' the space
$V$. Therefore for all $\tau$ we define $\Pi_\tau \tilde{w} =
\tilde{w} - f$, where $f\in V$ and $\tilde{w}-f$ is $L^2$-orthogonal
to $V$ at time $\tau$:
\[ \int_{M_\tau} (\tilde{w} - f) h\, e^{-|\x|/4}\, d\mathcal{H}^2 =
  0, \text{ for all } h\in V.\]
It was shown in \cite[Proposition 3.11]{LSSzAncient} that a three
annulus lemma analogous to Lemma~\ref{lem:3ann} holds for the norms
$\Vert \Pi_\tau\tilde{w}\Vert_\tau$ as well. 

Consider the normalized solutions
\[ \tilde{W}_i = \frac{\Pi_{\tau_i} \tilde{w}}{\Vert \Pi_{\tau_i}
    \tilde{w}\Vert_{\tau_i}} \]
of the drift heat equation. Translating in time, these define
solutions of the drift heat equation along $M^i_\tau :=
M_{\tau+\tau_i}$, with $L^2$-norm 1 at $\tau=0$. 
Using the three annulus lemma, arguing
similarly to \cite[Proposition 3.10]{LSSzAncient}, and in particular
using \cite[Proposition 3.7]{LSSzAncient}, we can extract a homogeneous
limit $\tilde{W}_\infty$
of these along a subsequence, defining a
solution of the drift heat equation on $P_1\cup_\ell P_2$. This
convergence is smooth on compact sets in $(0, \infty) \times
(\mathbb{C}^2 \setminus \ell)$, and also in $L^2$ for any $\tau > 0$.
We have the following.
\begin{lemma}
  The limit $\tilde{W}_\infty$ has growth rate at most $-1/2$,
  i.e. $\Vert \tilde{W}_\infty\Vert_\tau \leq Ce^{-\tau/2}$. 
\end{lemma}
\begin{proof}
  Suppose that $\tilde{W}_\infty$ had a faster growth rate. Then by homogeneity
  it would have to have degree 0, since on $P_1\cup_\ell P_2$ there
  are no homogeneous solutions with rates between $-1/2$ and 0
  (corresponding to linear functions and constants on the two planes). Since
  $\tilde{W}_\infty$ is the limit of the $\tilde{W}_i$ (translated in
  time), it follows that in this case along a subsequence of the
  $\tau_i$ we would have
  \[ \Vert \Pi_{\tau_i} \tilde{w}\Vert_{\tau_i+2} \geq e^{-1/4}  \Vert
    \Pi_{\tau_i} \tilde{w}\Vert_{\tau_i+1}. \]
  Once $i$ is sufficiently large, the three annulus lemma can be
  applied, and it implies that we have
  \[ \Vert \Pi_{\tau_i} \tilde{w}\Vert_{\tau_i+k+1} \geq e^{-1/4}  \Vert
    \Pi_{\tau_i} \tilde{w}\Vert_{\tau_i + k}, \]
  for all $k\geq 0$. In particular it would follow that
  \[ \Vert \tilde{w}\Vert_{\tau_i + k} \geq c e^{-k/4}, \]
  for some $c > 0$, and all sufficiently large $k$. This is a
  contradiction, since by the uniform area ratio bounds we have a
  constant $C > 0$ such that $\Vert \tilde{w}\Vert_\tau \leq C
  e^{-\tau/2}$ for all $\tau > 0$. 
\end{proof}

It is natural to expect that in a generic situation the limit
$\tilde{W}_\infty$ should have the largest possible growth rate. The
space of homogeneous solutions of the drift heat equation on
$P_1\cup_\ell P_2$ with rate $-1/2$ is spanned by $\tilde{x}, \tilde{y},
\tilde{z}$ and $\theta\tilde{z}$, where note that $\theta$ equals two
different constants on the two planes (see also \cite[Lemma
3.6]{LSSzAncient}). Since by construction $\tilde{W}_\infty$ is
orthogonal to $\tilde{x}, \tilde{y}, \tilde{z}$, we make the following
definition.

\begin{definition}\label{defn:ndteardrop}
 The teardrop singularity at $(\x_0, T_0)$ is \emph{nondegenerate} if
 $\tilde{W}_\infty = c \theta \tilde{z}$ for a constant $c\not=0$. 
\end{definition}

We now prove Theorem~\ref{thm:ndteardrop}, which we restate here.
\begin{thm}
  Suppose that $L_t$ admits a nondegenerate teardrop singularity at
  $(\x_0, T_0)$. Then the flow $L_t$ cannot be embedded in any
  neighborhood $B_r(\x_0) \times (T_0 - r^2, T_0)$. 
\end{thm}
\begin{proof}
  The proof follows more or less directly from \cite[Proposition
  4.5]{LSSzAncient}. By assumption we have a sequence $\tau_i\to
  \infty$ such that the translated rescaled flows $M^i_\tau :=
  M_{\tau+ \tau_i}$ converge weakly to the static flow $P_1\cup P_2$,
  and the normalized functions $\tilde{W}_i$, shifted in time,
  converge to $c\theta \tilde{z}$. Write $L^i_t$ for the mean
  curvature flows normalized so that $L^i_{-1} = M^i_0$. In particular
  the flows $L^i_t$ converge weakly to the static flow $P_1\cup P_2$
  as well. From the convergence of the $\tilde{W}_i$ it
  follows that for a sequence of constants $\sigma_i$, and functions
  $\phi_i\in \mathrm{span}(x,y,z)$ we have
  \[ \sigma_i^{-1} (w - \phi_i) \to z\theta, \]
  where the convergence is smooth on compact subsets of $[-1,0) \times
  (\mathbb{C}^2\setminus \{0\})$, and also in $L^2$.
  If we had $\sigma_i \not\to 0$, then along a subsequence
  we would have $\sigma_i^{-1} \phi_i \to z\theta$ since we  have $w
  \to 0$ (as $w$ vanishes on $P_1\cup P_2$). This is not possible,
  since the coordinate functions $x,y,z$ are $L^2$-orthogonal to
  $z\theta$ on $P_1\cup P_2$. Therefore we must have $\sigma_i\to 0$,
  which implies that we also have $\phi_i \to 0$.

  We next want to apply \cite[Proposition 4.5]{LSSzAncient}. In that
  paper we considered an exact ancient solution of LMCF, and a
  sequence of flows $L^i_t$ on $[-1,0)$ was obtained as a sequence of
  blowdowns at $-\infty$, converging to $P_1\cup P_2$. In our setting
  the flow is not assumed to be exact, but instead our original flow
  $L_t$ is rational. Our sequence of rescalings $L^i_t$ are such that
  $L^i_{-1} = M_{\tau_i} = e^{\tau_i/2}(L_{T_0 - e^{-\tau_i}} - \x_0),
  $
  and without loss of generality we can assume that the factors
  $e^{\tau_i/2}$ are all integers. It follows that then there will be
  primitives $\beta_i : L^i_t \to \mathbb{R} / 2\pi\mathbb{Z}$
  satisfying the equation
  \[ (\partial_t - \Delta) (\beta_i + 2t\theta_i) = 0 \]
  as in \cite[Lemma 4.1]{LSSzAncient}. While these $\beta_i$ are
  multivalued, the key result, \cite[Proposition 4.3]{LSSzAncient} uses
  $B_i = \cos(\beta_i + 2(t-s_1)\theta_i)$ (see loc. cit. for details
  on the notation), and this function is a real valued function on
  $L^i_t$. This means that the proof of \cite[Proposition
  4.5]{LSSzAncient} goes through for the sequence of rescalings $L^i_t$ as well,
  even though our flow is only rational rather than exact. The
  consequence is that in our setting the flow $L^i_t$ is not embedded
  in the ball $B_2(0)$, for some value of $t\in (-1,0)$. 
\end{proof}

The next step in
verifying Joyce's conjectural picture would be to use the existence of
an immersed point, as in the conclusion of the theorem, to show that
as $t\to T_0$, there is a holomorphic disk with boundary on $L_t$,
whose area converges to zero. This would show that such a
nondegenerate teardrop singularity can only occur if $L_t$ has
obstructed Floer homology just before the singular time, and we
should have done an earlier ``opening the nech'' surgery to prevent
this. For a more detailed explanation, see Joyce~\cite[Section 3.4,
Principle 3.29]{JoyceEMS}.

\bibliography{mybib}{}

\begin{thebibliography}{10}

\bibitem{Anc06}
Henri Anciaux.
\newblock Construction of {L}agrangian self-similar solutions to the mean
  curvature flow in {$\Bbb C^n$}.
\newblock {\em Geom. Dedicata}, 120:37--48, 2006.

\bibitem{BK23}
Richard~H Bamler and Bruce Kleiner.
\newblock On the multiplicity one conjecture for mean curvature flows of
  surfaces.
\newblock {\em arXiv:2312.02106}, 12 2023.

\bibitem{BM17}
Tom Begley and Kim Moore.
\newblock On short time existence of {L}agrangian mean curvature flow.
\newblock {\em Math. Ann.}, 367(3-4):1473--1515, 2017.

\bibitem{Bog98}
Vladimir~I. Bogachev.
\newblock {\em Gaussian measures}, volume~62 of {\em Mathematical Surveys and
  Monographs}.
\newblock American Mathematical Society, Providence, RI, 1998.

\bibitem{CCMS24}
Otis Chodosh, Kyeongsu Choi, Christos Mantoulidis, and Felix Schulze.
\newblock Mean curvature flow with generic initial data.
\newblock {\em Invent. Math.}, 237(1):121--220, 2024.

\bibitem{CCMS24_2}
Otis Chodosh, Kyeongsu Choi, Christos Mantoulidis, and Felix Schulze.
\newblock Revisiting generic mean curvature flow in $\mathbb{R}^3$.
\newblock {\em arXiv:2409.01463}, 09 2024.

\bibitem{CM12}
Tobias~H. Colding and William~P. Minicozzi, II.
\newblock Generic mean curvature flow {I}: generic singularities.
\newblock {\em Ann. of Math. (2)}, 175(2):755--833, 2012.

\bibitem{Ecker}
Klaus Ecker.
\newblock Logarithmic {S}obolev inequalities on submanifolds of {E}uclidean
  space.
\newblock {\em J. Reine Angew. Math.}, 522:105--118, 2000.

\bibitem{Fuk03}
Kenji Fukaya and Kenji Fukaya.
\newblock Galois symmetry on {F}loer cohomology.
\newblock {\em Turkish J. Math.}, 27(1):11--32, 2003.

\bibitem{GS09}
Zhou Gang and Israel~Michael Sigal.
\newblock Neck pinching dynamics under mean curvature flow.
\newblock {\em J. Geom. Anal.}, 19(1):36--80, 2009.

\bibitem{Hui90}
Gerhard Huisken.
\newblock Asymptotic behavior for singularities of the mean curvature flow.
\newblock {\em J. Differential Geom.}, 31(1):285--299, 1990.

\bibitem{Ilm95}
T.~Ilmanen.
\newblock Singularities of mean curvature flow of surfaces.
\newblock {\em preprint}, 1995.

\bibitem{INS19}
Tom Ilmanen, Andr\'{e} Neves, and Felix Schulze.
\newblock On short time existence for the planar network flow.
\newblock {\em J. Differential Geom.}, 111(1):39--89, 2019.

\bibitem{IJS16}
Yohsuke Imagi, Dominic Joyce, and Joana Oliveira~dos Santos.
\newblock Uniqueness results for special {L}agrangians and {L}agrangian mean
  curvature flow expanders in {$\Bbb{C}^m$}.
\newblock {\em Duke Math. J.}, 165(5):847--933, 2016.

\bibitem{JoyceEMS}
Dominic Joyce.
\newblock Conjectures on {B}ridgeland stability for {F}ukaya categories of
  {C}alabi-{Y}au manifolds, special {L}agrangians, and {L}agrangian mean
  curvature flow.
\newblock {\em EMS Surv. Math. Sci.}, 2(1):1--62, 2015.

\bibitem{JLT10}
Dominic Joyce, Yng-Ing Lee, and Mao-Pei Tsui.
\newblock Self-similar solutions and translating solitons for {L}agrangian mean
  curvature flow.
\newblock {\em J. Differential Geom.}, 84(1):127--161, 2010.

\bibitem{LSz24}
Yang Li and G{\'a}bor Sz{\'e}kelyhidi.
\newblock Singularity formations in lagrangian mean curvature flow.
\newblock 10 2024.

\bibitem{LN13}
Jason~D. Lotay and Andr\'{e} Neves.
\newblock Uniqueness of {L}angrangian self-expanders.
\newblock {\em Geom. Topol.}, 17(5):2689--2729, 2013.

\bibitem{LotayOliveira}
Jason~D. Lotay and Goncalo Oliveira.
\newblock Neck pinch singularities and joyce conjectures in lagrangian mean
  curvature flow with circle symmetry.
\newblock 05 2023.

\bibitem{LSSz22}
Jason~D. Lotay, Felix Schulze, and G{\'a}bor Sz{\'e}kelyhidi.
\newblock Neck pinches along the lagrangian mean curvature flow of surfaces.
\newblock 08 2022.

\bibitem{LSSzAncient}
Jason~D. Lotay, Felix Schulze, and G\'{a}bor Sz\'{e}kelyhidi.
\newblock Ancient solutions and translators of {L}agrangian mean curvature
  flow.
\newblock {\em Publ. Math. Inst. Hautes \'{E}tudes Sci.}, 140:1--35, 2024.

\bibitem{Neves07}
Andr\'{e} Neves.
\newblock Singularities of {L}agrangian mean curvature flow: zero-{M}aslov
  class case.
\newblock {\em Invent. Math.}, 168(3):449--484, 2007.

\bibitem{NevesSurvey}
Andr\'{e} Neves.
\newblock Recent progress on singularities of {L}agrangian mean curvature flow.
\newblock In {\em Surveys in geometric analysis and relativity}, volume~20 of
  {\em Adv. Lect. Math. (ALM)}, pages 413--438. Int. Press, Somerville, MA,
  2011.

\bibitem{Neves13}
Andr\'{e} Neves.
\newblock Finite time singularities for {L}agrangian mean curvature flow.
\newblock {\em Ann. of Math. (2)}, 177(3):1029--1076, 2013.

\bibitem{Pol91}
L.~Polterovich.
\newblock The surgery of {L}agrange submanifolds.
\newblock {\em Geom. Funct. Anal.}, 1(2):198--210, 1991.

\bibitem{SS20}
Felix Schulze and Natasa Sesum.
\newblock Stability of neckpinch singularities.
\newblock {\em arXiv:2006.06118}, 06 2006.

\bibitem{Sei99}
Paul Seidel.
\newblock Lagrangian two-spheres can be symplectically knotted.
\newblock {\em J. Differential Geom.}, 52(1):145--171, 1999.

\bibitem{Smoczyk}
Knut Smoczyk.
\newblock A canonical way to deform a lagrangian submanifold.
\newblock {\em arXiv:dg-ga/9605005}.

\bibitem{STW24}
Wei-Bo Su, Chung-Jun Tsai, and Albert Wood.
\newblock Infinite-time singularities of the lagrangian mean curvature flow.
\newblock {\em arXiv:2401.02228}, 01 2024.

\bibitem{SWX25}
Ao~Sun, Zhihan Wang, and Jinxin Xue.
\newblock Passing through nondegenerate singularities in mean curvature flows.
\newblock {\em arXiv:2501.16678}, 01 2025.

\bibitem{Thomas01}
R.~P. Thomas.
\newblock Moment maps, monodromy and mirror manifolds.
\newblock In {\em Symplectic geometry and mirror symmetry ({S}eoul, 2000)},
  pages 467--498. World Sci. Publ., River Edge, NJ, 2001.

\bibitem{TY02}
R.~P. Thomas and S.-T. Yau.
\newblock Special {L}agrangians, stable bundles and mean curvature flow.
\newblock {\em Comm. Anal. Geom.}, 10(5):1075--1113, 2002.

\bibitem{Whi94}
Brian White.
\newblock Partial regularity of mean-convex hypersurfaces flowing by mean
  curvature.
\newblock {\em Internat. Math. Res. Notices}, (4):186 ff., approx. 8 pp.\,
  1994.

\bibitem{White05}
Brian White.
\newblock A local regularity theorem for mean curvature flow.
\newblock {\em Ann. of Math. (2)}, 161(3):1487--1519, 2005.

\bibitem{Wol05}
Jon Wolfson.
\newblock Lagrangian homology classes without regular minimizers.
\newblock {\em J. Differential Geom.}, 71(2):307--313, 2005.

\bibitem{Wood24}
Albert Wood.
\newblock Singularities of equivariant {L}agrangian mean curvature flow.
\newblock {\em Comm. Anal. Geom.}, 32(7):1905--1952, 2024.

\end{thebibliography}
\bibliographystyle{plain}

\end{document}